\documentclass[11pt]{article}
\usepackage[letterpaper,twoside,outer=1.3in,vmargin=1.3in,]{geometry}
\usepackage{amsmath,amssymb,amsthm,amsfonts,enumerate,times,epsfig,color,hyperref,cleveref}
\usepackage{tikz}
\usetikzlibrary{shapes,positioning}

\usepackage{pifont}
\newcommand{\cmark}{\ding{51}}%
\newcommand{\xmark}{\ding{55}}%

\newcommand{\ignore}[1]{}

\DeclareMathOperator{\aff}{aff}
\DeclareMathOperator{\spn}{span}

\DeclareMathOperator{\supp}{support}

\DeclareMathOperator{\row}{row}

\DeclareMathOperator{\opt}{\mathrm{opt}}

\DeclareMathOperator{\vol}{\mathrm{vol}}

\newcommand{\rank}[1]{\mathrm{rank}\!\left(#1\right)}
\newcommand{\size}[1]{\mathrm{size}\!\left(#1\right)}

\newcommand{\1}{\mathbf{1}}
\newcommand{\0}{\mathbf{0}}

\newcommand{\bO}{\mathit O}
\newcommand{\Q}{\mathbb Q}
\newcommand{\R}{\mathbb R}
\newcommand{\Z}{\mathbb Z}

\newcommand{\cL}{\mathbb L}

\newcommand{\zS}{\mathcal S}

\newcommand{\plp}{p\text{-LP}}

\newtheorem{theorem}{Theorem}[section]
\newtheorem{CO}[theorem]{Corollary}
\newtheorem{LE}[theorem]{Lemma}

\newtheorem{CN}[theorem]{Conjecture}
\newtheorem{RE}[theorem]{Remark}

\newtheorem{EG}[theorem]{Example}
\newtheorem{DE}[theorem]{Definition}
\newtheorem*{MFT}{Minkowski's First Theorem}
\newtheorem*{JT}{Jacobi's Theorem}

\newtheorem{innercustomgeneric}{\customgenericname}
\providecommand{\customgenericname}{}
\newcommand{\newcustomtheorem}[2]{%
  \newenvironment{#1}[1]
  {%
   \renewcommand\customgenericname{#2}%
   \renewcommand\theinnercustomgeneric{##1}%
   \innercustomgeneric
  }
  {\endinnercustomgeneric}
}

\newcustomtheorem{customthm}{Theorem}
\newcustomtheorem{customLE}{Lemma}
\newcustomtheorem{customEG}{Example}

\newcounter{claim_nb}[theorem]
\newtheorem{claim}[claim_nb]{Claim}
\newtheorem*{claim*}{Claim}

\newcounter{claim_nbs}[section]
\newcounter{subclaim_nb}[claim_nbs]
\newenvironment{cproof}
{\begin{proof}
 [Proof of Claim.]
 \vspace{-1.2\parsep}}
{\renewcommand{\qed}{\hfill $\Diamond$} \end{proof}}

\title{Dyadic linear programming and extensions}

\author{Ahmad Abdi\thanks{Department of Mathematics, London School of Economics. Corresponding author, email address: a.abdi1@lse.ac.uk} \and
G\'erard Cornu\'ejols\thanks{Tepper School of Business, Carnegie Mellon University} \and
Bertrand Guenin\thanks{Department of Combinatorics \& Optimization, University of Waterloo} \and
Levent Tun\c{c}el\thanks{Department of Combinatorics \& Optimization, University of Waterloo}
}

\begin{document}

\maketitle


\begin{abstract}
A rational number is \emph{dyadic} if it has a finite binary representation $p/2^k$, where $p$ is an integer and $k$ is a nonnegative integer. Dyadic rationals are important for numerical computations because they have an exact representation in floating-point arithmetic on a computer. A vector is  \emph{dyadic} if all its entries are dyadic rationals. We study the problem of finding a dyadic optimal solution to a linear program, if one exists. We show how to solve dyadic linear programs in polynomial time.  We give bounds on the size of the support of a solution as well as on the size of the denominators.
We identify properties that make the solution of dyadic linear programs possible: closure under addition and negation, and density, and we extend the algorithmic framework beyond the dyadic case.\\

\end{abstract}
\section{Introduction}

A rational number is {\it dyadic} if it is an integer multiple of $\frac{1}{2^k}$ for some nonnegative integer $k$. Dyadic numbers are important
for numerical computations because they have a finite binary representation, and therefore they
can be represented exactly on a computer in floating-point arithmetic.
When real or rational numbers are approximated by dyadic numbers on a computer, approximation errors may propagate and accumulate throughout the computations. So it is natural to ask when linear programs have a dyadic optimal solution.
More generally, the exact solution of linear programs with rational input data has lead to interesting work; we just mention here the excellent dissertation of Espinoza \cite{espinoza}. In this paper we investigate a different direction.

A vector $x$ is {\it dyadic} if all its entries are dyadic numbers.
A \emph{dyadic linear program} is an optimization problem of the form $$
\sup\left\{w^\top x: Ax\leq b, x \text{ dyadic}\right\}
$$ where $A,b,w$ have integral entries.

Note that we do not restrict
ourselves to fixed precision; we just require a finite number of bits in the binary representation.  This is an important point as we will see that it makes the problem tractable. On the other hand, if the vector $x$ in the dyadic linear program were restricted to be of the form $\frac{y}{2^k}$ for an integral vector $y$ and a nonnegative integer $k$ bounded above by a given value $K$, then the problem would be a classical integer linear program. Indeed the problem can then be written as $\frac{1}{2^K}\cdot\max\left\{w^\top y: Ay\leq 2^K b, y \text{ integral}\right\}$.

Some natural questions about dyadic linear programs are: When is the problem feasible? Can we check feasibility in polynomial time? If the problem is infeasible, can we provide a certificate of infeasibility?
When does a dyadic linear program have an optimal solution? 
Note that a dyadic linear program may be feasible and bounded, but not have an optimal solution (in dimension one, $\sup\left\{x: 3x\leq 1, x \text{ dyadic}\right\}$ is such an example).
Can
dyadic linear programs be solved in polynomial time? What is the size of the dyadic numbers in a solution when one exists?  What is the support size of a solution? This paper addresses these questions. In particular, we show that dyadic linear programs can be solved in polynomial time.

The interest in dyadic linear programming stems not only from the computer science perspective mentioned above, but also from mathematics and from optimization. Take a mathematical point of view: Given a prime integer $p \geq 2$, we say that a rational number is {\em finitely $p$-adic} if it is of the form $\frac{r}{p^k}$ for some integer $r$ and nonnegative integer $k$. This concept is closely related to the notion of the \emph{$p$-adic numbers} introduced by Hensel, formally defined as the set of ``finite-tailed" infinite series $\sum_{i=N}^{+\infty} a_i p^i$ where $N\in \mathbb{Z}$, and $a_i \in \mathbb{Z}$ and $0 \leq a_i < p$ for each $i\geq N$.\footnote{The $p$-adic numbers form a field extension of the rational numbers, albeit in a different way than the real numbers.} The study of $p$-adic numbers gives rise to beautiful and powerful mathematics; see the excellent book by Gouv\^ea for more \cite{gouvea}. It can be readily checked that the set of finitely $p$-adic numbers is the set of finite series of the form $\sum_{i=N}^{M} a_i p^i$, where $M,N\in \mathbb{Z}$, $M\geq N$, and $0 \leq a_i < p, a_i \in \mathbb{Z}$ for all $N\leq i\leq M$, justifying our terminology.\footnote{The set of finitely $p$-adic numbers does not form a field. For instance, $3$ is finitely $2$-adic, but $1/3$ is not.}
In this paper, we will only deal with finitely $p$-adic numbers; for simplicity we refer to them as {\em $p$-adic numbers} throughout the paper. Also we refer to ``$2$-adic" as ``dyadic". More generally, we say that a rational number is \emph{$[p]$-adic} if it is of the form $\frac{r}{s}$ where $r$ is an integer and $s$ is a product of powers of primes between 2 and $p$. These numbers appear naturally in some of the theorems in this paper. For the optimization point of view, let us mention an intriguing conjecture of Seymour dating back to 1975; see Schrijver \cite{Schrijver03} 79.3e. Let $A$ be a 0,1 matrix such that the set covering polyhedron $Ax \geq {\bf 1}, x \geq 0$ has only integral vertices, where ${\bf 1}$ denotes the vector of all 1s. Thus the linear program
$\min\left\{c^\top x: Ax\geq {\bf 1}, x \geq 0\right\}$ has an optimal 0,1 solution for any objective function $c \in \mathbb{Z}^n_+$. Seymour conjectured that the dual linear program $\max\left\{{\bf 1}^\top y: A^\top y\leq c,  y \geq 0 \right\}$ always has a dyadic optimal solution $y$. This conjecture is still open but is known to hold in a few important special cases. For example, when $A$ is the $T$-cut versus edge incidence matrix of a graft, an optimal dual solution $y$ is $\frac{1}{2}$-integral  (Lov\'asz \cite{Lovasz1976}). Seymour's conjecture was proved recently in a couple of other special cases \cite{Abdi-dyadic}, \cite{Abdi-Tjoins}.

We will show in this paper that dyadic linear programming shares aspects of classical linear programming as well as certain aspects of integer programming. In particular, we will show that, just like linear programs, dyadic linear programs can be solved in polynomial time. However, when it comes to the support size of a solution, the situation is more akin to that in integer programming. Indeed, in classical linear programming, for a problem in standard equality form $x\geq 0,  Ax = b$, if it has an optimal solution, there is a basic optimal solution with at most $m$ nonzero entries where $m$ is the row rank of the constraint matrix $A$. For integer programs, the support size may be superlinear in $m$, and a similar situation occurs for dyadic linear programs.
Next, we present an outline of the paper.

In \Cref{section:foundation}, we present two key ideas that make the polynomial solution of dyadic linear programs possible. The first ingredient is the density of the dyadic numbers in the real line. This property enables us to convert the feasibility question for a dyadic linear program to that of the existence of a dyadic point in an affine space. Specifically, we show that, if $\cL$ is a dense subset of the real line closed under addition and negation, and $P$ is a nonempty convex set whose affine hull is rational, then $P$ contains a point in $\cL ^n$ if and only if  its affine hull $\aff(P)$ does. This equivalence begs the question: When does the affine space  $\aff(P)$ contain a point in $\cL ^n$?

The second key ingredient is a theorem of the alternatives. This theorem allows us to answer the above question.  Specifically, consider a matrix  $A\in\Z^{m\times n}$ and a vector $b\in \Z^m$. Then exactly one of the following holds: (a) $Ax = b$ has a solution in $\cL ^n$, or (b) there exists $u \in \mathbb{R}^m$ such that $A^\top u \in \mathbb{Z}^n$ and $b^\top u \not\in \cL$.
This theorem of the alternative is reminiscent of the so-called ``integer Farkas lemma" and it can be proved in a similar way, using the Hermite normal form of $A$. Because the Hermite normal form of an integral matrix can be found in polynomial time (Kannan and Bachem \cite{Kannan79}), one can obtain a polynomial certificate for statements (a) or (b), whichever holds.

These two are the basic ingredients we need to check feasibility of dyadic linear programs in polynomial time.

In \Cref{sec-optimization}, we present an algorithm to solve dyadic linear programs. There are four possible outcomes for this optimization problem: (i) the problem is infeasible, (ii) the problem is unbounded, (iii) the problem has an optimal solution, (iv) the problem is feasible and bounded but has no optimal solution. We show how to decide in polynomial time which of these outcomes is the correct one and, in each case, we give a concise (polynomial size) certificate. The complexity of our algorithm is (up to a constant factor) the same as that of solving an ordinary linear program. Our results in this section are extended beyond the dyadic numbers to any subset $\mathbb{L}\subseteq \mathbb{R}$ that is closed under addition and negation, contains all the $p$-adic numbers for some prime $p$, and is equipped with a membership oracle.

In \Cref{section:fractionality}, we focus on the size of the denominators of a solution to a feasible dyadic linear program. In particular, we show that if $Ax\leq b, x \text{ dyadic}$ is feasible, where $A\in \Z^{m\times n}$ and $b\in \Z^m$, then there exists a $\frac{1}{2^k}$-integral solution, where $k\leq \left\lceil\log_2 n+ (2n+1)\log_2(\|A\|_\infty\sqrt{n+1})\right\rceil$. Here $\|A\|_\infty$ denotes the largest absolute value of an entry in $A$.

In \Cref{section:support}, we study the size of the support of a solution to a dyadic linear program. That is, we consider the smallest number of nonzero components in a dyadic solution. Surprisingly, the lower and upper bounds that we obtain on the smallest support size resemble results for integer programming and are very different from the value of the support size given by Carath\'eodory's theorem for classical linear programming. Specifically, let $A\in \Z^{m\times n},b\in \Z^m$ and $w\in \R^n$. If $\min\{w^\top x:Ax=b,x\geq \0,x \text{ dyadic}\}$ has an optimal solution, then we show that it has one with support size at most $m(1+0.84\ln{m}+1.68\ln\|A\|_\infty)$. We give lower bounds by constructing examples and show that they are extremal in some sense. 

\Cref{section:conclusion} provides conclusions and possible directions for future research.
\section{Foundational results} \label{section:foundation}
In this section, we identify two results that underpin much of the work in the paper, and are key to solving dyadic linear programs. The first result reduces \emph{dyadic feasibility} of a rational polyhedron to that of its affine hull, while the second result provides a \emph{theorem of the alternatives} for dyadic feasibility of a linear system of equations. We then combine the two results to obtain a \emph{certificate of dyadic infeasibility} of a rational polyhedron. The results in this section are presented not only for the dyadic numbers, but more generally for any subset $\cL$ of $\R$ closed under addition and negation, where sometimes we require $\cL$ to be dense in $\R$, and at other times $\cL\neq \R$. Let us lay the groundwork.

\subsection{Abelian subgroups, $p$-adic and $[p]$-adic numbers}\label{sec-subgroup}

Let $\cL\subseteq \R$. We say that $\cL$ is {\em closed under addition} if $x+y\in \cL$ for all $x,y\in\cL$, and that $\cL$ is {\em closed under negation}, or equivalently {\em symmetric around the origin}, if $-x\in \cL$ for all $x\in\cL$. Clearly, if $\cL$ is closed under addition and negation, then $ax\in \cL$ for all $a\in \Z$ and $x\in\cL$. In particular, if $\cL\neq \emptyset$ then $0\in \cL$.

The condition that $\cL\subseteq \R$ is closed under addition and negation is equivalent to requiring that $(\cL,+)$ forms an abelian subgroup of $(\R,+)$. It can be readily checked that for any such nonempty set, either $\cL$ is dense in $\R$, or $\cL$ is of the form $\{ax:a\in\Z\}$ for some $x\in\R$. In the latter case, $\cL$ is isomorphic to $\Z$, in which case checking if a rational polyhedron in $\R^n$ contains a point in $\cL^n$ is equivalent to integer programming feasibility in dimension $n$. 

The discussion above motivates us to focus on dense sets in $\R$ closed under addition and negation. We present a construction for such sets, but we first require a definition. A subset $\zS$ of positive integers is {\em closed under multiplication} if $pq\in\zS$ for all $p,q\in\zS$. Given such a subset, note that $\zS$ is finite if and only if $\zS\subseteq \{1\}$.
\begin{LE}\label{L-construction}
Let $\zS$ be an infinite set of positive integers that is closed under multiplication. Let $\cL$ be the set of all numbers $\frac{r}{s}$ where $s\in\zS$ and $r\in\Z$. Then $\cL$ is dense in $\R$ and closed under addition and negation.
\end{LE}
\begin{proof}
Pick $x\in\R$ and $\epsilon>0$.
Since $\zS$ is an infinite set of positive integers, there exists $s\in\zS$ with $s\geq\frac{1}{\epsilon}$.
Let $r=\lfloor sx\rfloor$, then $|x-\frac{r}{s}|<\epsilon$. 
Since $\frac{r}{s}\in\cL$, it follows that $\cL$ is a dense subset of $\R$.
Since $\zS$ is closed under multiplication, if $\frac{r}{s},\frac{r'}{s'}\in\cL$, then their sum $\frac{rs'+r's}{ss'}\in\cL$, so $\cL$ is closed under addition.
Finally, if $\frac{r}{s}\in\cL$ then $\frac{-r}{s}\in\cL$, so $\cL$ is closed under negation.
\end{proof}

We can use this result to provide two important examples of dense sets in $\R$ closed under addition and negation.

\begin{DE}
Let $p\geq 2$ be a prime number. We say that a rational number is {\em $p$-adic} if it is of the form $\frac{r}{s}$ where $s$ is a power of $p$ and $r\in\Z$. We say that a rational number is {\em $[p]$-adic} if it is of the form $\frac{r}{s}$ where $s$ is a product of powers of primes between $2$ and $p$, and $r\in\Z$.
\end{DE}

Observe that the set of $p$-adic numbers can be obtained from the construction in \Cref{L-construction} by choosing $\zS=\{p^n:n\geq\0,n\in\Z\}$, while the set of $[p]$-adic numbers is obtained by choosing $\zS$ to consist of all positive integers with prime factors less than or equal to $p$. Thus, both sets are dense in $\R$ and closed under addition and negation, by \Cref{L-construction}. Observe further that when $p=2$, the $p$-adic, $[p]$-adic, and dyadic numbers coincide. More generally, we have the following relation between $p$-adic and $[p]$-adic sets.
\begin{theorem}
Let $p\geq 2$ be an integer. Then the following are equivalent for a subset $\cL\subseteq\R$:
\begin{enumerate}[\;\;(1)]
\item
$\cL$ is the set of $[p]$-adic numbers,
\item 
$\cL$ is an inclusion-wise minimal set closed under addition and negation that contains all $q$-adic numbers for primes $q\leq p$.
\end{enumerate}
\end{theorem}
\begin{proof}
Denote by $\cL_1$ the set $\cL$ as defined by (1) and by $\cL_2$ the set $\cL$ as defined by (2). We need to show $\cL_1=\cL_2$. 
Note that $\cL_2\subseteq\cL_1$ since
(i) every $q$-adic number for a prime $q\leq p$ is $[p]$-adic by definition, and since
(ii) $\cL_1$ is closed under addition and negation by \Cref{L-construction}.
We need to show that every point in $\cL_1$ is in $\cL_2$.
Since $\cL_2$ is closed under addition and negation, it suffices to consider points of the form $\frac1{Q}\in\cL_1$ where $Q\in\Z$, $Q\geq 1$.
As $1\in\cL_2$ we may assume $Q\geq 2$.
Then $Q=\prod_{i=1}^rq_i^{\alpha_i}$ where $q_1,\ldots,q_r\leq p$ are distinct primes and $\alpha_i\geq 1$, $\alpha_i\in\Z$ for all $i\in[r]$.
For all $i\in[r]$ define, $Q_i:=Qq_i^{-\alpha_i}$.
Then $Q_1,\ldots,Q_r$ are relatively prime and so it follows from B\'{e}zout's lemma that there exist $\rho_1,\ldots,\rho_r\in\Z$ such that $\sum_{i=1}^r\rho_iQ_i=1$. 
Observe that,
\[
\sum_{i=1}^r \rho_i q_i^{-\alpha_i}=\frac{1}{Q}\sum_{i=1}^r\rho_iQ_i=\frac{1}{Q}.
\]
Each of the terms on the left hand side of the previous expression is in $\cL_2$.
Since $\cL_2$ is closed under addition, it follows that $\frac{1}{Q}\in\cL_2$, as required. 
\end{proof}

\subsection{Density and affine hulls}\label{sec-density}
In \cite{TDD, acgt23} we proved that a rational polyhedron contains a dyadic point if and only if its affine hull does. This is a special case of the following more general result. 

\begin{theorem}\label{density-gen}
Let $\cL$ be a dense subset of $\R$ that is closed under addition and negation. 
Consider $P\subseteq\R^n$ where (i) the relative interior of $P$ is non-empty, and (ii) the affine hull $\aff(P)$ of $P$ is a translate of a rational subspace.
Then $P\cap\cL^n\neq\emptyset$ if and only if $\aff(P)\cap\cL^n\neq\emptyset$.
\end{theorem}

To prove this theorem, we need a technical lemma. To state it we need some notations. For a vector $v$ and $1\leq q\leq+\infty$, $\|v\|_q$ denotes the $q$-norm of $v$. Given $\epsilon>0$ and a vector $\bar{x}$, we write $B_q(\bar{x},\epsilon)$ for the closed $q$-norm ball $\left\{x:\|x-\bar{x}\|_q\leq\epsilon\right\}$ centred at $\bar{x}$ with radius $\epsilon$.
\begin{LE}\label{find-xip}
Let $\cL$ be a dense subset of $\R$ that is closed under addition and negation, and let $P\subseteq\R^n$.
Suppose that $\aff(P)=z+\spn\{d^1,\ldots,d^\ell\}$ for some $z\in\R^n$ and $d^1,\ldots,d^\ell\in\Z^n$, in particular, $\aff(P)$ is a translate of a rational subspace. Consider $z'\in \aff(P)$ and $\epsilon>0$. 
Then the following statements hold for any $1\leq q\leq +\infty$:
\begin{enumerate}[\;\;(a)]
\item There exists $\bar{\rho}\in\cL^n$ such that $z+\bar{\rho}\in\aff(P)\cap B_q(z',\epsilon)$.
\item
If for some prime $p$, $\cL$ contains all the $p$-adic numbers, we can find in (a) an explicit $\bar{\rho}$ that is $p$-adic.
Namely, pick a nonnegative integer $r$ for which
\begin{equation}\label{choice-r}
p^r\geq\frac{\ell\max\{\|d^1\|_q,\ldots,\|d^\ell\|_q\}}{\epsilon}.
\end{equation}
Pick $\alpha\in \R^\ell$ satisfying 
\begin{equation}\label{choice-alpha}
z+\sum_{i=1}^\ell\alpha_id^i=z'. 
\end{equation}
Then we can choose
\begin{equation}\label{choice-rho}
\bar{\rho} := \sum_{i=1}^\ell \frac{\lfloor p^r\alpha_i\rfloor}{p^r}d^i.
\end{equation}
Furthermore, for this choice of $\bar{\rho}$ we have $\|p^r\bar{\rho}\|_q\leq p^r(\|z'-z\|_q+\epsilon)$.
\end{enumerate}
\end{LE}
\begin{proof}
Given $\beta\in \R^\ell$ denote $
\rho(\beta):=\sum_{i=1}^\ell\beta_i d^i.
$ Then, since $z' \in \aff(P)$, we have $z+\rho(\alpha)=z'$ for some $\alpha\in \R^\ell$. We need the claim below.
\begin{claim*}
$z+\rho(\beta)\in B_q(z',\epsilon)$ if 
\begin{equation}\label{beta-alpha}
|\beta_i-\alpha_i|\leq\frac{\epsilon}{\ell\max\{\|d^1\|_q,\ldots,\|d^\ell\|_q\}}
\qquad \text{ for } i=1,\ldots,\ell.
\end{equation}
\end{claim*}
\begin{cproof}
We need to show that $\|z+\rho(\beta)-z'\|_q\leq\epsilon$. We have,
\begin{align*}
\| z+\rho(\beta)-z'\|_q = & \|\rho(\beta)-\rho(\alpha)\|_q =
\left\|\sum_{i=1}^\ell(\beta_i-\alpha_i) d^i\right\|_q \\
\leq &
\sum_{i=1}^\ell \left\|(\beta_i-\alpha_i) d^i\right\|_q = 
\sum_{i=1}^\ell |\beta_i-\alpha_i| \; \|d^i\|_q \leq \epsilon,
\end{align*}
where the first inequality arises from the triangle inequality and the second inequality from \eqref{beta-alpha}.
\end{cproof}

{\bf (a)}
Since $\cL$ is dense in $\R$, we can pick $\beta_i\in\cL, i\in [\ell]$ such that \eqref{beta-alpha} holds. Let $\bar{\rho}:=\rho(\beta)$.
Since $\cL$ is closed under addition and negation, and since $\beta_i\in\cL$ and $d^i$ is integral for each $i\in [\ell]$, we have $\bar{\rho}\in\cL^n$.
By the claim, $z+\bar{\rho}\in B_q(z',\epsilon)$.
Moreover, since $\aff(P)=z+\spn\{d^1,\ldots,d^\ell\}$, we have $z+\bar{\rho}\in\aff(P)$.
Thus, $z+\bar{\rho}\in \aff(P)\cap B_q(z',\epsilon)$, so (a) holds.

{\bf (b)}
Suppose now $\cL$ contains all the $p$-adic numbers. 
Pick $r$ so that \eqref{choice-r} holds and for each $i=1,\ldots,\ell$ let $\beta_i:=\frac{\lfloor p^r\alpha_i\rfloor}{p^r}$.
Then $|\beta_i-\alpha_i|\leq\frac1{p^r}$. 
Equation \eqref{choice-r} then implies that \eqref{beta-alpha} holds.
It follows from the claim that for $\bar{\rho}:=\rho(\beta)$, we have $z+\bar{\rho}\in\aff(P)\cap B_q(z',\epsilon)$.

Moreover, as $\beta_i$ is $p$-adic and $d^i$ is integral for each $i\in [\ell]$, $\bar{\rho}$ is $p$-adic.
By the triangle inequality, $\|\bar{\rho}\|_q\leq \|z'-z\|_q+\|\bar{\rho}+z-z'\|_q$.
Multiplying both sides by $p^r$ yields $\|p^r\bar{\rho}\|_q\leq p^r\|z'-z\|_q+p^r\|\bar{\rho}+z-z'\|_q$.
As $\bar{\rho}+z\in B_q(z',\epsilon)$, we have $\|\bar{\rho}+z-z'\|_q\leq\epsilon$, so $\|p^r\bar{\rho}\|_q\leq p^r(\|z'-z\|_q+\epsilon)$.
\end{proof}

\Cref{find-xip}~(b) is used in \Cref{sec-optimization} for efficiently finding a dyadic point inside a rational polyhedron. As for part (a), we will use it below to prove \Cref{density-gen}.

\begin{proof}[Proof of \Cref{density-gen}]
Since $P\subseteq\aff(P)$, $\aff(P)\cap\cL^n=\emptyset$ implies $P\cap\cL^n=\emptyset$.
Assume now that $\aff(P)\cap\cL^n\neq\emptyset$. Pick $z\in\aff(P)\cap\cL^n$.
Since $P$ is a translate of a rational subspace by hypothesis (ii), we can express $\aff(P)$ as $z+\spn\{d^1,\ldots,d^\ell\}$ for some vectors $d^1,\ldots,d^{\ell}\in \Z^n$. 
By hypothesis (i), there exists a $p$-norm ball of $B$ of radius $\epsilon>0$ centred at some $z'\in P$ for which $\aff(P)\cap B\subseteq P$.
By \Cref{find-xip}(a) there exists $\bar{\rho}\in\cL^n$ for which $z+\bar{\rho}\in\aff(P)\cap B$.
Since $z,\bar{\rho}\in\cL^n$ and $\cL$ is closed under addition, $z+\bar{\rho}\in\cL^n$. Putting it altogether, we get that $z+\bar{\rho}\in P\cap \cL^n$, so $P\cap \cL^n\neq \emptyset$, as required.
\end{proof}

Let us present an important corollary of this result. Since every non-empty convex set has a non-empty relative interior, \Cref{density-gen} implies the following.

\begin{CO}\label{density-convex}
Let $\cL$ be a dense subset of $\R$ that is closed under addition and negation. Consider a nonempty convex set $P$, where $\aff(P)$ is a translate of a rational subspace.
Then $P\cap\cL^n\neq\emptyset$ if and only if $\aff(P)\cap\cL^n\neq\emptyset$.
\end{CO}

When $\cL$ is the set of the dyadic numbers, and $P$ is a polyhedron, we retrieve the fact that a rational polyhedron contains a dyadic point if and only if its affine hull does.

\subsection{Theorem of the alternatives and consequences}
In light of \Cref{density-convex} we are interested in characterizing when a rational affine space contains a point in $\cL^n$.
We addressed the case when $\cL$ is the set of dyadic points in \cite{acgt23}.
Density is irrelevant in this case however, and the following general theorem holds; its proof is a careful adaptation of the well-known result for the case of integers (e.g., see Theorem 1.17 in \cite{ConfortiCornuejolsZambelli2014}). For the proof, we need a definition. A square matrix is \emph{unimodular} if it has integral entries and its determinant is $\pm 1$. Observe that if $U$ is unimodular then so is $U^{-1}$, by Cramer's rule for instance.

\begin{theorem}\label{alternative}
Let $\cL$ be a proper subset of $\R$ that is closed under addition and negation.
Consider a matrix $A\in\Z^{m\times n}$ and a vector $b\in \Z^m$. 
Then exactly one of the following holds:
\begin{enumerate}[\;\;(a)]
\item $Ax=b$ has a solution in $\cL^n$,
\item there exists $u\in\R^m$ such that $A^\top u\in\Z^n$ and $b^\top u\notin\cL$.
\end{enumerate}
Moreover, if (a) holds, then $Ax=b$ has a solution in $\cL^n\cap\Q^n$.
\end{theorem}
\begin{proof}
Suppose (a) and (b) both hold. Then we have $\bar{x}\in\cL^n$ with $A\bar{x}=b$ and therefore $u^\top A\bar{x}=u^\top b$.
As $u^\top A$ is integral and since $\cL$ is closed under addition and negation, $\cL\ni u^\top A\bar{x}=u^\top b$, a contradiction.
Assume now that (a) does not hold. We prove that (b) holds.
If $Ax=b$ does not have a solution in $\Q^n$, then there exists $u\in \Q^m$ for which $u^\top A=\0^\top$ and $u^\top b\neq 0$.
Since $\cL$ is a proper subset of $\R$, there exists $p\in\R\setminus\cL$. Then after updating $u:=p\cdot u$, we have $u^\top A=\0^\top\in \Z^n$ and $u^\top b=p\notin \cL$, and (b) holds.
Thus let us assume that $Ax=b$ has a solution in $\Q^n$.
We may further assume that the rows of $A$ are linearly independent for otherwise we can eliminate redundant constraints,
prove (b) for the smaller system and derive the result for the original system.
We can now find a unimodular matrix $U\in \Z^{n\times n}$ for which $AU=(B\;\0)$ where $B$ is a square non-singular matrix (for example one can choose $U$ that converts $A$ into Hermite normal form, see \S\ref{sec:primer-lattices} for more).
We claim that $\bar{z}:=B^{-1}b\notin\cL^m$. 
For otherwise,
\begin{equation}\label{set-barx}
\bar{x}:=U\begin{pmatrix}\bar{z} \\ \0\end{pmatrix}, 
\end{equation}
is a solution to $Ax=b$.
Since $\cL$ is closed under addition and negation and since $U$ is integral, this would imply $\bar{x}\in\cL^n$, a contradiction as (a) does not hold by assumption.
Thus for some $i\in[m]$ we have $\bar{z}_i\notin\cL$.
Let $u=B^{-\top}e_i$. Then 
$u^\top b = e_i^{\top} B^{-1}b=\bar{z}_i\notin\cL$.
Moreover, $u^\top A = e_i^\top B^{-1}(B\;\0)U^{-1}\in\Z^n$ since $U^{-1}$ is integral as $U$ is unimodular.
Hence, (b) holds.

Thus, exactly one of (a) and (b) holds. For the final part of the theorem, suppose that (a) holds.
Then $\bar{z}=B^{-1}b\in\cL^m$, for otherwise the above argument shows that (b) holds, which is not the case. Furthermore, $\bar{z}\in \Q^m$, since $b\in \Z^m$, and $B$ is integral so $B^{-1}$ is rational.
Subsequently, $\bar{x}$ as defined in \eqref{set-barx} is a solution of $Ax=b$ in $\cL^n\cap\Q^n$, as required.
\end{proof}

There are two caveats with \Cref{alternative}. The first one is in the condition that $\cL\neq \R$; the only place in the proof where this condition is necessary is for finding the certificate $u$ in (b) in case (a) does not hold. The second caveat is that if (a) does not hold, then we cannot always find a certificate $u$ in (b) that is rational.
Consider for instance the case where $\cL$ contains all rationals.
Then we require $b^\top u\not\in\cL\supseteq\Q$ in (b), in particular, $u$ is not rational. However, in this example, if (a) does not hold, then $Ax=b$ has no solution in $\Q^n$ either. By eliminating this possibility we can guarantee that the certificate $u$ in (b) is rational.

\begin{RE}\label{rational-cert}
In \Cref{alternative}, if $Ax=b$ has a solution in $\Q^n$ and has no solution in $\cL^n$, then there exists a rational certificate $u$ in (b).
\end{RE}
\begin{proof}
Indeed, suppose that $Ax=b$ has a solution in $\Q^n$. 
Then, proceeding as in the proof of \Cref{alternative}, we can pick $u=B^{-\top}e_i$ for (b), which is rational.
\end{proof}

Given \Cref{density-convex} and the second part of \Cref{alternative}, we may ask the following question. Given a dense subset $\cL$ of $\R$ closed under addition and negation, and given a rational polyhedron $R$ which contains a point in $\cL^n$, does $R$ necessarily contain a point in $\cL^n\cap\Q^n$?
Alas, this need not be the case. Consider for instance $\cL:=\left\{\frac{a}{b}\sqrt{2}:a,b\in\Z, b\neq 0\right\}$, and let $R\subseteq\R^2$ be the convex hull of $(1,1)$ and $(2,2)$. Then $R\cap\cL^2\neq \emptyset$, but $\cL^2\cap \Q^2=(0,0)$ so $R\cap \cL^2\cap \Q^2=\emptyset$. Note that in this example, $\Q\cap \cL$ is not a dense subset of $\R$, and this is no accident as the next result shows.
\begin{theorem}\label{QinR}
Let $\cL$ be a subset of $\R$ that is closed under addition and negation and suppose that $\cL \cap \Q$ is dense in $\R$.
Then a rational polyhedron contains a point in $\cL^n$ if and only if it contains a point in $\cL^n\cap\Q^n$.
\end{theorem}
\begin{proof}
The result is well-known when $\cL=\R$. Otherwise, $\cL\neq \R$.
$(\Leftarrow)$ is clear. For $(\Rightarrow)$, suppose we have a rational polyhedron $P$ with $P\cap\cL^n\neq\emptyset$.
Then $\aff(P)\cap\cL^n\neq\emptyset$, so it follows from \Cref{alternative} that $\aff(P)\cap\cL^n\cap\Q^n\neq \emptyset$. Since $\cL\cap \Q$ is a dense subset of $\R$ that is closed under addition and negation, we may apply \Cref{density-convex} to $\cL\cap \Q$ and $P$ to conclude that $P\cap\cL^n\cap\Q^n\neq \emptyset$, as required.
\end{proof}

\subsection{Certificate of $\cL$-infeasibility of a rational polyhedron}\label{sec-char-feasibility}
Consider a dense subset $\cL$ of $\R$ that is closed under addition and negation, and let $P$ be a rational polyhedron. We wish to characterize when $P\cap\cL^n\neq\emptyset$. By \Cref{density-convex}, it suffices to check if $\aff(P)\cap\cL^n\neq\emptyset$.
Furthermore, since $P$ is a rational polyhedron, $\aff(P)$ is a rational affine space, so we can use \Cref{alternative} to 
characterize when $\aff(P)\cap\cL^n\neq\emptyset$. 
In this subsection we combine these results to yield a \emph{certificate $\cL$-infeasibility} of a rational polyhedron. We need the following technical lemma (which is also useful later in \Cref{sec-optimization}).

\begin{LE}\label{dyadic-face}
Let $\cL$ be a dense proper subset of $\R$ that is closed under addition and negation, and 
consider a polyhedron $P := \{x:Ax\leq b\}$ where $A\in\Z^{m\times n}$ and $b\in \Z^m$. Suppose that
\begin{equation}\tag{P}
\max\{c^\top x:Ax\leq b\},
\end{equation}
has an optimal solution of value $\tau$, and let $F:=P\cap\{x:c^\top x=\tau\}$ be the optimal face. Then the following statements are equivalent:
\begin{enumerate}[\;\;(a)]
\item $F\cap\cL^n=\emptyset$.
\item There exists $\bar{y},\bar{u}\in\Q^m$ that satisfy the following conditions,
\begin{enumerate}[(i)]
\item $\supp(\bar{u})\subseteq\supp(\bar{y})$,
\item $\bar{y}$ is an optimal solution to the dual of (P),
\item  $A^\top u\in\Z^n$ and $b^\top u\notin\cL$.
\end{enumerate}
\end{enumerate}
%
Morever, for $\bar{y}$ satisfying (ii), we have $\supp(\bar{y})\subseteq\{i\in[m]:\row_i(A)x=b_i,\;\mbox{for all}\;x\in F\}$.
\end{LE}
\begin{proof}
Let $I^= := \{i\in[m]:\row_i(A)x=b_i,\;\mbox{for all $x\in F$}\}$ and denote, by $A^=x\leq b^=$ the inequalities of $Ax\leq b$ corresponding to $I^=$.
Then, (see \cite{ConfortiCornuejolsZambelli2014}, Theorem 3.24), $F=P\cap\{x:A^=x=b^=\}$.
Therefore, (see \cite{ConfortiCornuejolsZambelli2014}, Theorem 3.17), $\aff(F)=\{x:A^=x=b^=\}$.
Consider the dual of (P),
\begin{equation}\tag{D}
\min\{b^\top y:A^\top y=c, y\geq\0\}.
\end{equation}

{\bf (a)$\Rightarrow$(b)}
By \Cref{density-convex}, $\aff(F)$ has no point in $\cL^n$.
Hence, by \Cref{alternative}, there exists $\bar{u}$ satisfying (iii) where $\supp(\bar{u})\subseteq I^=$.
Since $A^=x=b^=$ has a solution, it follows from \Cref{rational-cert} that we can choose $\bar{u}\in\Q^{I^=}$.
Let $\bar{x}$ and $\bar{y}$ be a strictly complementary pair of rational optimal solutions for (P) and~(D).
Then $\supp(\bar{y})=I^=$. It follows that (i) holds.

{\bf (b)$\Rightarrow$(a)}
Let $i\notin I^=$.
Then for some $\bar{x}\in F$ we have $\row_i(A)\bar{x}<b_i$.
By Complementary Slackness, $\bar{y}_i=0$, i.e. $i\notin\supp(\bar{y})$.
Hence, $\supp(\bar{y})\subseteq I^=$, and the ``moreover" statement holds.
Furthermore, by (i) we have $\supp(\bar{u})\subseteq I^=$.
Therefore, by \Cref{alternative} and (iii), there is no point in $\cL^n$ that also lies in $\{x:A^=x=b^=\}=\aff(F)$.
It follows that $F\cap\cL^n=\emptyset$.
\end{proof}
\begin{theorem}\label{infeasibility-certificate}
Let $\cL$ be a dense proper subset of $\R$ that is closed under addition and negation.
Consider a  non-empty polyhedron $P := \{x:Ax\leq b\}$ where $A\in\Z^{m\times n}$ and $b\in \Z^m$.
Then the following are equivalent.
\begin{enumerate}[\;\;(a)]
\item $P\cap\cL^n=\emptyset$.
\item There exists a {\bf certificate of $\cL$-infeasibility} for $P$, that is, a pair of vectors $\bar{y},\bar{u}\in\Q^m$ that satisfy the following conditions:
\begin{enumerate}[(i)]
\item $\supp(\bar{u})\subseteq\supp(\bar{y})$,
\item $\bar{y}\geq\0$, $A^\top\bar{y}=\0$, $b^\top\bar{y}=0$,
\item $A^\top u\in\Z^n$ and $b^\top u\notin\cL$.
\end{enumerate}
\end{enumerate}
%
Moreover, for $\bar{y}$ satisfying (ii), we have $\supp(\bar{y})\subseteq\{i\in[m]:\row_i(A)x=b_i,\;\mbox{for all}\;x\in P\}$.
\end{theorem}
\begin{proof}
Consider the linear program (P) defined as $\max\{0:Ax\leq b\}$ and let $F=P$. 
Every $\bar{x}\in P$ is an optimal solution to (P) of value $0$.
Therefore, condition (ii) says that $\bar{y}$ is an optimal solution to the dual of (P).
The result now follows from \Cref{dyadic-face}.
\end{proof}
\section{Algorithms: feasibility and optimization}\label{sec-optimization}
In this section, we present algorithms to check whether a rational polyhedron contains a point in $\cL^n$ and to optimize a linear function over a rational polyhedron restricted to $\cL^n$
where $\cL$ satisfies the following three properties:
\begin{enumerate}[\;\;\;\;\;]
\item (group) \quad\quad\quad $\cL$ is closed under addition and negation, 
\item (density) \quad\quad\; $\cL$ contains all $p$-adic numbers for some prime $p$, and
\item (membership) \; we have a membership oracle for $\cL$.
\end{enumerate}
Assuming the (group) property we also need to require $\cL$ to be a dense subset of the reals for otherwise we are operating within the context of integer programming as discussed in \Cref{sec-subgroup}. 
We want to be able to carry all computations over $\Q$, in particular, we insist that if a rational polyhedron contains a point in $\cL^n$ then it contains one in $\cL^n\cap\Q^n$.
Because of \Cref{QinR} this can be achieved by choosing $\cL$ with the property that $\cL \cap \Q$ is dense.
A natural choice is to include all $p$-adic numbers in $\cL$, i.e. imposing the (density) condition above.
Note, that we will assume that we are also given the prime value $p$ explicitly.
A {\em membership oracle} for $\cL$ is a function that takes as input $x\in\Q$ and returns {\sc yes} if $x\in\cL$ and {\sc no} otherwise. 
We will present algorithms that run in oracle polynomial time for both feasibility and optimization over polyhedra restricted to $\cL^n$.
When $\cL$ is the set of $p$-adic, or $[p]$-adic numbers, our three properties (group), (density), and (membership) hold, and we have trivial polynomial oracles.
Thus, we will be able to solve the feasibility and optimization problems for $p$-adic and $[p]$-adic numbers in polynomial time.
In particular, we can solve the problem for dyadic numbers.

Our algorithms will rely on the existence of a number of subroutines that we present next. 
  
\medskip\noindent{\bf Algorithm A.}\\
Takes as input a matrix $A\in \Q^{m\times n}$ and a vector $b\in\Q^m$.
Returns one of the following:
(i) $\bar{x}\in\Q^n$ satisfying $Ax=b$, or
(ii) $u\in\Q^m$ such that $A^\top u=\0$  and $b^\top u\neq 0$ certifying $Ax=b$ has no solution.

\medskip\noindent{\bf Algorithm B.}\\
Takes as input a matrix $A\in\Z^{m\times n}$ and a vector $b\in\Z^m$.
We are also given a membership oracle that describes a set $\cL\subseteq\R$ that is closed under addition and negation.
Returns one of the following:
(i) $\bar{x}\in\cL^n\cap\Q^n$ satisfying $Ax=b$, or 
(ii)  $u\in\Q^m$ such that $A^\top u=\0$  and $b^\top u\neq 0$ certifying $Ax=b$ has no solution, or
(iii) $u\in\Q^m$ such that $A^\top u$ is integral and $u^\top b\notin\cL$ certifying $Ax=b$ has no solution in $\cL^n$.

\medskip\noindent{\bf Algorithm C.}\\
Takes as input a matrix $A\in \Z^{m\times n}$.
Returns linearly independent vectors $d^1,\ldots,d^\ell\in\Z^n$ with the property that $\spn\{d^1,\ldots,d^\ell\}=\{x:Ax=\0\}.$

 \medskip\noindent{\bf Algorithm D.}\\
Takes as input a matrix $A\in \Q^{m\times n}$ and vectors $b\in\Q^m$ and $c\in\Q^n$.
Then solves the linear program
\begin{equation}\tag{P}
\max\{c^\top x:Ax\leq b\}.
\end{equation}
\noindent
Namely, it returns one of the following:
(i) $u\in\Q^m$ such that $A^\top u\geq\0$ and $b^\top u<0$ certifying that (P) is infeasible, or
(ii) $x,r\in\Q^n$ such that $Ax\leq b$, $Ar\leq\0$ and $c^\top r>0$ certifying that (P) is unbounded, or
(iii)  $x\in\Q^n$, $y\in\Q^m$ which form a pair of strictly complementary optimal solutions for (P) and its dual.
\medskip

Algorithm B emulates the argument in the proof of \Cref{alternative}. 
We first ensure that $Ax=b$ has a solution using Algorithm A, and then eliminate redundant constraints.
We then find a unimodular matrix $U\in\Z^{n\times n}$ for which $AU=(B\;\0)$ where $B$ is an $m$-by-$m$ matrix and compute $\bar{z}=B^{-1}b$.
We use the membership oracle to check whether each of $\bar{z}_i\in\cL$.
If this is the case then $U\begin{pmatrix}\bar{z} \\ \0\end{pmatrix}\in\cL^n$ is a solution to $Ax=b$.
Otherwise, for some $i\in[m]$, $\bar{z}_i\notin\cL$ and we return $u=B^{-\top}e_i$.
For Algorithm C, we eliminate linearly dependent rows of $A$ and find a unimodular matrix $U\in\Z^{n\times n}$ for which $AU=(B\;\0)$ where $B$ is an $m$-by-$m$ matrix.
Then the columns of $U$ corresponding to the $\0$ columns of $(B\;\0)$ are the required vectors $d^i$.
For both algorithms, finding the matrix $U$ can achieved by rewriting $A$ in Hermite normal form which can be found in polynomial time~\cite{Kannan79}.
This implies that there is an implementation of Algorithm~C that runs in polynomial time, and an implementation of Algorithm~B that runs in {\em oracle} polynomial time.
\subsection{$\cL$-Feasibility Problem (LFP)}\label{sec:real-feasibility}
Consider a polyhedron $P := \{x:Ax\leq b\}$. 
The {\em $\cL$-Feasibility Problem (LFP)} takes as input a matrix $A\in\Z^{m\times n}$ and a vector $b\in\Z^m$ that define the polyhedron~$P$.
We consider $\cL\subseteq\R$ satisfying the (group), (density) and (membership) properties (with $p$ given explicitly).
We then need to return 
(i) a point in $P\cap\cL^n$, or (ii) a certificate that $P=\emptyset$, or (iii) a certificate of $\cL$-infeasibility (as defined in \Cref{sec-char-feasibility}). We will show that there exists an oracle polynomial time algorithm to solve LFP. Note that by Farkas' lemma, $P=\emptyset$ if and only if there exists $y\geq\0$ for which $A^\top y=\0$ and $b^\top y<0$; such a $y$ is a {\em certificate of real-infeasibility}, and is the output of (ii).

\subsubsection{The LFP algorithm}\label{sec-lfp}
Here is a description of our algorithm to solve LFP.

\medskip\noindent{\em {\bf Step 1}: Find the implicit equalities of $P$.}
\medskip

\noindent
Consider the following primal-dual pair,
\begin{align*}
& \max\{0: Ax\leq b\}, \tag{P} \\
& \min\{b^\top y:A^\top y=\0, y\geq\0\}. \tag{D}
\end{align*}
Use Algorithm D to check whether $P=\emptyset$.
If it is, return a certificate of real-infeasibility $u$ and stop.
Otherwise, Algorithm D finds a strictly complementary pair of optimal solutions $\bar{x}\in\Q^n$ and $\bar{y}\in\Q^m$ for (P) and (D) respectively.
Since every $x\in P$ is an optimal solution to (P), strict complementarity implies that
\begin{claim*}
$\supp(\bar{y})=I^=$ where $I^==\{i\in[m]:\row_i(A)x=b_i\;\mbox{$\forall$ $x\in P$}\}$.
\end{claim*}
\noindent
Denote by $A^=x\leq b^=$ the constraints from $Ax\leq b$ indexed by $I^=$. Then $\aff(P)=\{x:A^=x=b^=\}$.

\medskip\noindent{\em {\bf Step 2}: Find an internal description of $\aff(P)$ or certify $\cL$-infeasibility.}
\medskip

\noindent
Let $I^<=[m]\setminus I^=$ and let $A^<x\leq b^<$ denote the constraints of $Ax\leq b$ indexed by $I^<$.
Then use Algorithm B to either,
\begin{enumerate}[\;\;(i)]
\item find $z\in\aff(P)\cap\cL^n\cap\Q^n$, or
\item find $u^=\in\Q^m$ such that ${A^=}^\top u^=$ is integral and $b^\top u^=\notin\cL$.
\end{enumerate}
If (ii) occurs, then extend $u^=$ to a vector indexing the rows of $A$ by assigning zeros to all entries corresponding to $I^<$, 
and denote the resulting vector by $\bar{u}$.
Then observe that $\bar{y},\bar{u}$ is a certificate of $\cL$-infeasibility, and we can stop.
Otherwise we have $z$ as described in (i).
Use Algorithm C to find an integral basis $d^1,\ldots,d^\ell$ of $\{d:A^=d=\0\}$.
It then follows that
\begin{equation}\label{aff-description}
\aff(P)=z+\spn\{d^1,\ldots,d^\ell\}.
\end{equation}

\noindent{\em {\bf Step 3}: Find a Euclidean ball contained in $P$.}
\medskip

\noindent
For each $i\in[m]$ let $\gamma_i:=\lceil\|\row_i(A)\|_2\rceil$ ($\gamma_i$ can be
computed in polynomial time on a Turing machine, without evaluating the
square-root to a high accuracy).
Consider the following linear program with variables $\zeta$ and $e$:
\begin{align}
\max\;\;\;\;\;\;\; & \;\;\;\;e  \notag\\
\mbox{subject to} \notag \\
& A^=\zeta=b^= \notag \\
& \gamma_i e+\row_i(A)\zeta\leq b_i & (i\in I^<) & \label{ball} \\
& e\leq 1. \notag
\end{align}
Since $e\leq 1$, \eqref{ball} is not unbounded. Moreover, \eqref{ball} is feasible (pick $\zeta\in P$ and $e=0$).
Hence, we can use Algorithm D to find an optimal solution $\zeta=z'\in\Q^n$, $e=\epsilon\in\Q$ of \eqref{ball}.
By definition of $I^<$ we have $\epsilon>0$.
Then $z'\in P$ and for every $i\in[m]$, $\frac{1}{\gamma_i}\left[b_i-\row_i(A)z'\right]$ is a lower bound on the Euclidean distance from $z'$ to the hyperplane $\{h:\row_i(A)h=b_i\}$ because $\gamma_i\geq\|\row_i(A)\|_2$. 
This implies that $B\subseteq\{x:A^<x\leq b^<\}$ where $B$ denotes the Euclidean ball of radius $\epsilon$ centred at $z'$.
Then $B\cap\aff(P)\subseteq\{x:A^<x\leq b^<\}\cap\{x:A^=x=b^=\}=P$.

\medskip\noindent{\em {\bf Step 4}: Find a point in $P\cap\cL^n$.}
\medskip

\noindent
Pick 
\begin{equation}\label{choose-r-dlp}
r := \left\lceil\log_p\left(\frac{\ell\max\{\|d^1\|_2,\ldots,\|d^\ell\|_2\}}{\epsilon}\right)\right\rceil.
\end{equation}
Then we satisfy \eqref{choice-r} (this is the place where we need to know the value of $p$ explicitly).
Use Algorithm~A to find $\alpha$ satisfying $\sum_{i\in[\ell]}\alpha_i d^i=z'-z$ so that \eqref{choice-alpha} holds.
Choose $\bar{\rho}$ as in \eqref{choice-rho}.
It then follows from \Cref{find-xip} that $\bar{\rho}$ is $p$-adic and that $z+\bar{\rho}\in \aff(P)\cap B\subseteq P$.
Finally, as $z\in\cL^n\cap\Q^n$, the (group) and (density) properties imply that $z+\bar{\rho}\in\cL^n\cap\Q^n$, as required.

\subsubsection{The alternate LFP algorithm}\label{sec-lfp-alt}
We sketch out a variant of our algorithm for solving the $\cL$-feasibility problem for polyhedra that are in standard equality form, i.e. for $P:=\{x\geq\0:Ax=b\}$.
The implicit equalities will be of the form $Ax=b$ as well as $x_j=0$ for a subset $J^=\subseteq[n]$ of the column indices.
In particular, we have 
\[
\aff(P)=\{x:Ax=b, x_j=0 \;\forall j\in J^=\}.
\]
\noindent
We consider the following primal-dual pair,
\begin{align*}
& \max\{0: Ax=b,x\geq\0\}, \tag{P'} \\
& \min\{b^\top y:A^\top y\geq \0\}. \tag{D'}
\end{align*}
In {\em Step 1}, we proceed as in the previous version and use Algorithm D to find a strictly complementary pair of optimal solutions $x'$ and $y'$ for (P') and (D').
Then $J^==\{j\in[n]: \mbox{col}_j(A)^\top y'>0\}$ and this allows us to identify $\aff(P)$.
In {\em Step 2}, we use Algorithm C find an integral basis $\{d^1,\ldots,d^\ell\}$ of $\{d:Ad=\0, d_j=0\;\forall j\in J^=\}$.
The most notable change is for {\em Step 3}. 
Let $J^<=[n]\setminus J^=$ and denote by $D$ the column submatrix of $A$ indexed by columns $J^<$.
Instead of \eqref{ball} we solve, the linear program with variables $e,\zeta$ where
\begin{align}
\max\;\;\;\;\;\;\; & \;\;\;\;e  \notag \\
\mbox{subject to} \notag \\
& D\zeta \, = \, b \notag\\
& e \leq \zeta_j & (j\in J^<) \label{ball-alternate}\\
& e\leq 1. \notag
\end{align}
The above linear program is clearly feasible and bounded, hence we can use Algorithm D to find an optimal solution $\zeta=\bar{z}$ and $e=\epsilon$.
Let $z'\in\Q^n$ be obtained from $\bar{z}$ by setting entries corresponding to $J^=$ to zero.
Let $B=\{z:\|z'-z\|_\infty\leq\epsilon\}$, i.e. $B$ is the $\infty$-norm ball of radius $\epsilon$ centred at $z'$.
Then we have $B\subseteq \{x: x_j\geq\0 \;\forall j\in J^<\}$ by construction.
It follows that $B\cap\aff(P)\subseteq\{x: x_j\geq 0 \;\forall j\in J^<\}\cap\{x:Ax=b, x_j=0 \;\forall j\in J^=\}=P$.
{\em Step 4.} is the same as for the original algorithm, except that while the original algorithm used the $2$-norm ball $B$ in \Cref{find-xip}, 
this version of the algorithm uses the $\infty$-norm ball. In particular, we need to choose $r$ according to,
\begin{equation}\label{choose-r-dlp-inf}
r := \left\lceil\log_p\left(\frac{\ell\max\{\|d^1\|_\infty,\ldots,\|d^\ell\|_\infty\}}{\epsilon}\right)\right\rceil.
\end{equation}
rather than according to \eqref{choose-r-dlp}.
\subsubsection{Output size and running time analysis}
We refer to the algorithm described in \Cref{sec-lfp} as \emph{the LFP algorithm}, and the algorithm described in \Cref{sec-lfp-alt} as the {\em alternate} LFP algorithm.
Informally, the runtime of our LFP algorithm is dominated by the runtime of our linear program solver, 
which we call twice, once for checking feasibility of $Ax\leq b$ and once for solving \eqref{ball}.
Note, that both linear programs have essentially the same size (with respect to many measures), moreover, they are both using the original data $A,b$.
Thus, checking for $\cL$-feasibility is about at most twice as time consuming as checking for real-feasibility. 
The same conclusion applies to the alternate LFP algorithm.

We will analyze the running time of the LFP algorithm and the output size of the alternate LFP algorithm (the analysis is somewhat cleaner for that version because of the use of $\infty$-norms).
The {\em encoding size} of an integer $\alpha$ is defined as $\size{\alpha}:=\log_2(\alpha)$. The encoding size of a rational $\frac{r}{s}$ is defined as $\size{\frac{r}{s}}:=\size{r}+\size{s}$.
Observe that for rationals $\alpha,\beta$ we have $\size{\alpha}=\size{\frac1{\alpha}}$ and that $\size{\alpha\beta}\in\bO\bigl(\size{\alpha}+\size{\beta}\bigr)$.
We consider $p$ which appears in the (density) property of the set $\cL$ to be an absolute constant.
For each of Algorithms A-D we can view the input as an $m'$-by-$n'$ matrix where the largest encoding size of any entry is given by~$\sigma$.
Then the runtime is bounded by a function $f_a$, $f_b$, $f_c$, $f_d$ for each of algorithms A, B, C, D where $f_a$, $f_b$, $f_c$, $f_d$ are functions of $m',n',\sigma$.
In addition, for the output, the largest entry size is bounded by a function $g_b$, $g_c$, $g_d$ for each of algorithms B, C, D where $g_b$, $g_c$, $g_d$ are also functions of $m', n', \sigma$.
There exist implementations for each algorithm where each of $f_a, f_b, f_c, f_d, g_b, g_c, g_d$ is a polynomial in $m',n',\sigma$.
Note, for Algorithm B the runtime is in terms of the number of calls to the $\cL$-oracle.

Let us now analyze the runtime.
\begin{theorem}
Consider $A\in\Z^{m\times n}$ and a vector $b\in \Z^m$ and let $\sigma$ denote the largest size of any entry in $A$ or $b$.
Then the LFP algorithm has runtime
\[
\bO\bigl( 
f_a(n,n+1,\sigma')+
f_b(m,n+1,\sigma)+
f_c(m,n,\sigma)+
f_d(m+n+1,n+2,\sigma)
\bigr)
\]
where
\[
\sigma'=g_b(m,n+1,\sigma)+g_c(m,n,\sigma)+g_d(m+n+1,n+2,\sigma).
\]
\end{theorem}
\begin{proof}
The runtime of the algorithm is longest when it finds a point in $P\cap\cL^n$ so we assume that this is the case in the analysis.
{\em Step~1.}
The runtime is dominated by Algorithm~D with $A,b$. 
It is in $\bO\bigl(f_d(m,n+1,\sigma)\bigr)$ since the data $A,b$ can be represented as an $m$-by-$(n+1)$ matrix.
The algorithm returns $\bar{x}$ and $\bar{y}$ which are used to find the implicit equalities $A^=x\leq b^=$ of $Ax\leq b$.
{\em Step~2.}
We run Algorithm~B for the system $A^=x=b^=$. 
This takes $\bO\bigl(f_b(m,n+1,\sigma)\bigr)$ and we obtain $z\in\cL^n\cap\Q^n$ with largest size of an entry in $\bO\bigl(g_b(m,n+1,\sigma)\bigr)$. 
We then use Algorithm~C to get an integral basis $\{d^1,\ldots,d^\ell \}$ of $\{d:A^=d=\0\}$. Note that $\ell\leq n$.
This takes $\bO\bigl(f_c(m,n,\sigma)\bigr)$ and for each $i\in[\ell]$ the largest size of an entry in $d^i$ is $\bO\bigl(g_c(m,n,\sigma)\bigr)$.
{\em Step~3.}
The linear program \eqref{ball} has $n+1$ variables and at most $m+n+1$ constraints.
We run Algorithm~D to solve it.
This takes $\bO\bigl(f_d(m+n+1,n+2,\sigma)\bigr)$ and we return $z'$ and $\epsilon$ with largest size of an entry in $\bO\bigl(g_d(m+n+1,n+2,\sigma)\bigr)$.
{\em Step~4.}
We solve the system $\sum_{i\in[\ell]}\alpha_i d^i=z'-z$ using Algorithm~A. 
Note, that $z$, $d^i$, $z'$ are outputs of Algorithms B, C and D respectively.
As $\ell\leq n$, this takes $\bO\bigl(f_a(n,n+1,\sigma')\bigr)$ for $\sigma'$ defined as in the statement.
\end{proof}

We close this discussion by showing that functions $g_b, g_c, g_d$ applied to the original data determine the size of the solution returned by the alternate LFP algorithm.
\begin{theorem}
Consider $A\in\Z^{m\times n}$ and a vector $b\in \Z^m$ and let $\sigma$ denote the largest size of any of an entry in $A$ or $b$.
Assume that $P:=\{x\geq\0:Ax=b\}$ contains a point in $\cL^n$. 
Then the alternate LFP algorithm finds a point $\bar{x}\in P\cap\cL^n\cap\Q^n$ with 
\[
\size{\|\bar{x}\|_\infty}\in\bO\bigl(g_b(m+n,n+1,\sigma)+g_c(m+n,n,\sigma)+g_d(m+n+1,n+1,\sigma) + \log(n)\bigr). 
\]
\end{theorem}
\begin{proof}
In Step 1, we find $J^=$ with $\aff(P)=\{x:Ax=b,x_j=0\;\forall j\in J^=\}$.
Thus $\aff(P)$ is described by a system with at most $m+n$ constraints.
In Step~2, we find $z\in\aff(P)$ using Algorithm~B. Hence,
\begin{equation}\label{size-z}
\size{\|z\|_\infty}\in\bO\bigl(g_b(m+n,n+1,\sigma)\bigr).
\end{equation}
We then use Algorithm~C to find an integer basis $d^1,\ldots,d^\ell$ of $\{d:Ad=\0, d_j=0\;\forall j\in J^=\}$. Thus,
\begin{equation}\label{size-d}
\size{\|d^i\|_\infty}\in\bO\bigl(g_c(m+n,n,\sigma)\bigr).
\end{equation}
In Step~3, we solve \eqref{ball-alternate} using Algorithm D. It follows that,
\begin{equation}\label{size-e-z'}
\size{\epsilon},\size{\|z'\|_\infty}\in\bO\bigl(g_d(m+n+1,n+1,\sigma)\bigr).
\end{equation}
In Step~4 we pick $r$ as in \eqref{choose-r-dlp-inf}. Note, that $\ell\leq n$. Therefore,
\[
\size{p^r}\in\bO\bigl(\size{\|d\|_\infty}+\size{\epsilon}+\size{n}\bigr).
\]
By \eqref{size-d} and \eqref{size-e-z'} in follows in turn that,
\begin{equation}\label{size-qr}
\size{p^r}\in\bO\bigl(g_c(m+n,n,\sigma)+g_d(m+n+1,n+1,\sigma) + \log(n) \bigr).
\end{equation}
In Step~4 we construct $\bar{\rho}$ as described in \eqref{choice-rho}.
Since $p=\infty$, \Cref{find-xip}(b) implies that for every $i\in[n]$
\begin{equation}\label{bound-qrrho}
|p^r\bar{\rho}_i| \leq p^r(|z'_i-z_i|+1).
\end{equation}
Since $p^r\bar{\rho}_i\in\Z$ it follows from \eqref{bound-qrrho} that 
$\size{p^r\bar{\rho}_i}\in\bO\bigl(\size{p^r}+\size{\|z\|_\infty}+\size{\|z'\|_\infty}\bigr)$. 
This in turn implies,
\[
\size{\|\bar{x}=z+\bar{\rho}\|_\infty}\in\bO\bigl(\size{p^r}+\size{\|z\|_\infty}+\size{\|z'\|_\infty}\bigr).
\]
The result now follows from \eqref{size-z}, \eqref{size-e-z'} and \eqref{size-qr}.
\end{proof}
%
\subsection{$\cL$-Linear Programming}
Suppose that $\cL$ is dense and closed under addition and negation, and consider the following optimization problem,
\begin{equation}\label{dlp}
\sup\left\{c^\top x:Ax\leq b,\;x\in\cL^n\right\},
\end{equation}
where $A\in\Z^{m\times n}$, $b\in\Z^m$ and $c\in\Z^n$.
We say that \eqref{dlp} is an {\em $\cL$-linear program}.
Trivially, the status of this optimization problem has to fall within exactly one of the following outcomes:
\begin{enumerate}[\;\;(o1)]
\item it is infeasible,
\item it is unbounded,
\item it has an optimal solution,
\item it  is feasible, bounded, but has no optimal solution.
\end{enumerate}
Note that in contrast to linear programming, (o4) can occur, such as for $
\sup\{x : 3x \leq 1,\;\mbox{$x$ dyadic}\}$.
\subsubsection{Certificates}
Here we certify each of the outcomes (o1)-(o4).
We certified (o1) with a {\bf certificate of $\cL$-infeasibility} in \Cref{infeasibility-certificate}. 
(o2) is similar to linear programming.
\begin{theorem}\label{unbounded}
Let $\cL\supset\Z$ be a dense subset of $\R$ that is closed under addition and negation.
Suppose that \eqref{dlp} is feasible.
Then the following are equivalent:
\begin{enumerate}[\;\;(a)]
\item \eqref{dlp} is unbounded,
\item $\max\{c^\top x:Ax\leq b\}$ is unbounded,
\item there exists $r\in\Z^n$ with $Ar\leq\0$ and $c^\top r>0$.
\end{enumerate}
A {\bf certificate of unboundedness} is a pair $(\bar{x},r)$, where $\bar{x}$ is a feasible solution to \eqref{dlp}, and $r$ is from~(b). 
\end{theorem}
\begin{proof}
Denote by (P) the linear program $\max\{c^\top x:Ax\leq b\}$.
Since \eqref{dlp} is feasible, so is (P).
For the equivalence between (b) and (c), see Proposition 3.9 in \cite{ConfortiCornuejolsZambelli2014}.
Note, that for (c) the condition that $r$ is rational is equivalent to the condition that $r$ is integral because of scaling.
We show that (a) and (b) are equivalent.
First if \eqref{dlp} is unbounded, so is (P) since it is a relaxation of \eqref{dlp}.
Thus assume that (P) is unbounded.
Then there exists $r$ as described in (c).
Let $\bar{x}$ be a feasible solution to \eqref{dlp} and for any $\lambda$ define $x(\lambda):=\bar{x}+\lambda r$.
Consider $\lambda\in\Z$ where $\lambda\geq 0$. Then $\lambda r\in\Z^n$.
By the hypothesis $\cL\supset\Z$, it follows that $\lambda r\in\cL^n$.
Since $\cL$ is closed under addition and negation, $x(\lambda)\in\cL^n$ and is therefore feasible for \eqref{dlp}.
For $\lambda\rightarrow\infty$ we have $c^\top x(\lambda)\rightarrow\infty$.
Hence, \eqref{dlp} is unbounded.
\end{proof}

We next distinguish between outcomes (o3) and (o4).
\begin{theorem}\label{o3-versus-o4}
Let $\cL$ be a dense subset of $\R$ that is closed under addition and negation.
Suppose that \eqref{dlp} is feasible and bounded.
Then the following are equivalent,
\begin{enumerate}[\;\;(a)]
\item \eqref{dlp} has no optimal solution,
\item there exist $\bar{x}\in\R^n$ and $\bar{y}, \bar{u}\in\R^m$ that satisfy the following conditions,
\begin{enumerate}[(i)]
\item $\supp(\bar{u})\subseteq\supp(\bar{y})$,
\item $A\bar{x}\leq b$, $A^\top\bar{y}=c$, $\bar{y}\geq\0$, $c^\top\bar{x}=b^\top\bar{y}$,
\item $A^\top\bar{u}$ is integral and $b^\top\bar{u}\notin\cL$.
\end{enumerate}
\end{enumerate} A {\bf certificate of unattainability} is a triple $(\bar{x}, \bar{y}, \bar{u})$ satisfying (i)-(iii).
\end{theorem}
\begin{proof}
Condition (ii) implies that $\bar{x}$ and $\bar{y}$ form a pair of primal dual solutions with the same value, hence, that $\bar{y}$ is optimal for the dual of (P).
Conversely, if $\bar{y}$ is optimal, then by strong duality there exists an optimal solution $\bar{x}$ of (P) with $c^\top\bar{x}=b^\top\bar{y}$.
Thus condition (ii) simply says that $\bar{y}$ is optimal for the dual.
The result now follows from \Cref{dyadic-face}.
\end{proof}

We can be more specific about (o4). When \eqref{dlp} is feasible and bounded, then we say that a feasible solution $\bar{x}\in\cL^n$ is an \emph{$\epsilon$-approximation} if $c^\top\bar{x}\geq\max\{c^\top x:Ax\leq b\}-\epsilon$. The result below shows that for every $\epsilon>0$ there exists an $\epsilon$-approximation.
\begin{theorem}\label{epsilon}
Let $\cL\supset\Z$ be a dense subset of $\R$ that is closed under addition and negation.
If \eqref{dlp} is feasible and bounded, then,
\[ 
\sup\left\{c^\top x:Ax\leq b,\;x\in\cL^n\right\}=\max\{c^\top x:Ax\leq b\}. 
\]
\end{theorem}
\begin{proof}
Denote by (P) the linear program $\max\{c^\top x:Ax\leq b\}$.
Since \eqref{dlp} is feasible so is (P).
Since \eqref{dlp} is bounded, so is (P) by \Cref{unbounded}.
It follows that (P) has an optimal solution, say $\bar{x}$ and let $\tau:=c^\top\bar{x}$.
Pick $\epsilon>0$ and define $Q_\epsilon:=\{x:Ax\leq b\}\cap\{x:c^\top x\geq\tau-\epsilon\}$.
Because of $\bar{x}$, $c^\top x\geq\tau-\epsilon$ is not an implicit equality of $Q_\epsilon$.
Therefore, $\{x:Ax\leq b\}$ and $Q_\epsilon$ have the same implicit equalities, say $A^=x\leq b^=$.
Hence, $\aff(\{x:Ax\leq b\})=\aff(Q_\epsilon)=\{x:A^=x=b^=\}$.
Since \eqref{dlp} is feasible, there exists a solution to $A^=x=b^=$ in $\cL^n$.
Hence, by \Cref{density-convex} there exists a point in $Q_\epsilon\cap\cL^n$.
In particular, $\sup\{c^\top x:Ax\leq b,\;x\in\cL^n\}$ has a feasible solution of value $\geq \tau-\epsilon$.
Letting $\epsilon\rightarrow 0$ proves the result.
\end{proof}

\subsubsection{Outcomes for primal-dual pairs}
Here we review the possible outcomes for a primal-dual pair of $\cL$-linear programs.
Namely, we consider the following pair of optimization problems,
\begin{align*}
& \sup\bigl\{c^\top x:Ax\leq b,\;x\in\cL^n\bigr\}, \tag{P} \\
& \inf\bigl\{b^\top y:A^\top y=c,\;y\geq\0, y\in\cL^n\bigr\} \tag{D}
\end{align*}
where $\cL\supset\Z$ is dense in $\R$ and closed under addition and negation.
We observed that each of (P) and (D) has 4 possible outcomes (o1)-(o4). 
In the next table we indicate the possible pairs of outcomes for (P) and (D).
\begin{table}[htb]
\begin{center}
\begin{tabular}{c||c|c|c|c|}
& (o1) & (o2) & (o3) & (o4) \\ \hline
(o1) & \cmark & \cmark & \cmark & \cmark \\ \hline
(o2) & \cmark & \xmark & \xmark  & \xmark \\ \hline
(o3) & \cmark & \xmark & \cmark & \cmark \\ \hline
(o4) & \cmark & \xmark & \cmark  & \cmark \\ \hline
\end{tabular}
\end{center}
\vspace{-0.1in} \caption{\small Possible outcomes for (P) and (D).\label{table-outcomes}}
\end{table}
The rows of the matrix indicate the possible outcomes for (P) and the columns of the matrix indicate the possible outcomes for (D).
A check mark in the table indicates that the corresponding outcome is possible, and a cross that the outcome is not possible.
We illustrate some of these outcomes. Consider
\begin{align}
&\sup\{3x:3x=1,x\geq\0, x\;\mbox{dyadic}\}, \label{o1o3} \\
&\sup\{x:3x=1,x\geq\0,x\;\mbox{dyadic}\}, \label{o1o4} \\
&\sup\{x:3x\leq 3,x\geq\0,x\;\mbox{dyadic}\},\label{o3o4}\\
&\sup\{x:3x\leq 1,x\geq\0,x\;\mbox{dyadic}\} \label{o4o4}.
\end{align}
Then for (P) described as \eqref{o1o3} and its dual (D) we have outcomes (o1) for (P) and (o3) for (D).
Similarly, 
\eqref{o1o4} corresponds to row (o1) and column (o4);  
\eqref{o3o4} corresponds to row (o3) and column (o4);  and
\eqref{o4o4} corresponds to row (o4) and column (o4).  
\Cref{table-outcomes}  is also applicable to pairs of primal-dual convex optimization  problems in conic form 
(although the details of their duality theory are significantly  more complicated than our setting, from this very specific and high level view, they coincide).

We close this discussion by including a table that indicates the possible pairs of outcomes for (P) and the LP relaxation of (D).
The rows of the matrix indicate the possible outcomes for (P) and the columns of the matrix indicate the possible outcomes for the  LP relaxation of (D).
\begin{table}[h]
\begin{center}
\begin{tabular}{c||c|c|c|}
& (o1) & (o2) & (o3) \\ \hline
(o1) & \cmark & \cmark & \cmark  \\ \hline
(o2) & \cmark & \xmark & \xmark  \\ \hline
(o3) & \xmark & \xmark & \cmark  \\ \hline
(o4) & \xmark & \xmark & \cmark   \\ \hline
\end{tabular}
\end{center}
\vspace{-0.1in}\caption{\small Possible outcomes for (P) and the LP relaxation of its dual (D).}
\end{table}
\vspace{-0.1in}
\subsubsection{Algorithm for solving $\cL$-linear programs}
We consider $\cL$ satisfying the (group), (density) and (membership) properties (with $p$ given explicitly).
Observe that $\cL\supset\Z$ and $\cL$ satisfies the hypotheses of \Cref{unbounded}, \Cref{o3-versus-o4}, and \Cref{epsilon}, 
that characterize possible outcomes, and $\epsilon$-approximation for $\cL$-linear programs.

We are given $\epsilon>0$, integral matrices and vectors $A,b,c$ describing the $\cL$-linear program \eqref{dlp}.
Then in oracle polynomial time we return one of the following, 
(i) a certificate of real-infeasibility, 
(ii) a certificate of $\cL$-infeasibility,
(iii) a certificate of unboundedness, 
(iv) an optimal solution,
(v) a certificate of unattainability and an $\epsilon$-approximation. 
Note, if we know that we have one of outcomes (o1), (o2), (o3) then it is not necessary to supply $\epsilon$. 
A schematic representation of the algorithm is given in \Cref{solver-description}.
\begin{figure}[htb]
\begin{center}
\begin{tikzpicture}[font=\small,thick]
\node[draw,
    rectangle,
    minimum width=4.5cm,
    minimum height=1cm] (block1) {Check $\cL$-feasibility using LFP Algorithm};
  
\node[draw,
    rounded rectangle,
    right=2cm of block1,
    minimum width=2.5cm,
    minimum height=1cm] (block5) {(o1)};  
    
\node[draw,
    rectangle,
    below=of block1,
    minimum width=2.5cm,
    minimum height=1cm] (block2) { Algorithm D};

\node[draw,
    rounded rectangle,
    right=3.75cm of block2,
    minimum width=2.5cm,
    minimum height=1cm] (block6) {(o2)};
 
\node[draw,
    rectangle,
    below=of block2,
    minimum width=2.5cm,
    minimum height=1cm] (block3) { LFP algorithm on the optimal face};
    
\node[draw,
    rounded rectangle,
    right=2.45cm of block3,
    minimum width=2.5cm,
    minimum height=1cm] (block7) { (o3)}; 
 
 \node[draw,
    rectangle,
    below=of block3,
    minimum width=2.5cm,
    minimum height=1cm] (block4) { LFP algorithm on $F'$}; 
 
\node[draw,
    rounded rectangle,
    below=of block4,
    minimum width=2.5cm,
    minimum height=1cm] (block8) { (o4)};

 \draw[-latex] (block1) edge node[pos=0.4,fill=white,inner sep=2pt]{No}(block5)
    node[pos=0.25,fill=white,inner sep=0]{};
    
     \draw[-latex] (block1) edge node[pos=0.4,fill=white,inner sep=2pt]{$x^f$}(block2)
    node[pos=0.25,fill=white,inner sep=0]{};

 \draw[-latex] (block2) edge node[pos=0.4,fill=white,inner sep=2pt]{LP unbounded}(block6)
    node[pos=0.25,fill=white,inner sep=0]{};
    
\draw[-latex] (block2) edge node[pos=0.4,fill=white,inner sep=2pt]{strictly complementary pair}(block3)
 node[pos=0.25,fill=white,inner sep=0]{};    
     
 \draw[-latex] (block3) edge node[pos=0.4,fill=white,inner sep=2pt]{Yes}(block7)
    node[pos=0.25,fill=white,inner sep=0]{};
    
 \draw[-latex] (block3) edge node[pos=0.4,fill=white,inner sep=2pt]{certificate of unattainability}(block4) node[pos=0.25,fill=white,inner sep=0]{};        

 \draw[-latex] (block4) edge node[pos=0.4,fill=white,inner sep=2pt]{$\epsilon$-approximation $\bar{x}$}(block8)
    node[pos=0.25,fill=white,inner sep=0]{};
 
\end{tikzpicture}
\end{center}
\vspace{-0.1in}\caption{Schematic description of the $\cL$-linear program solver \label{solver-description}}
\end{figure}
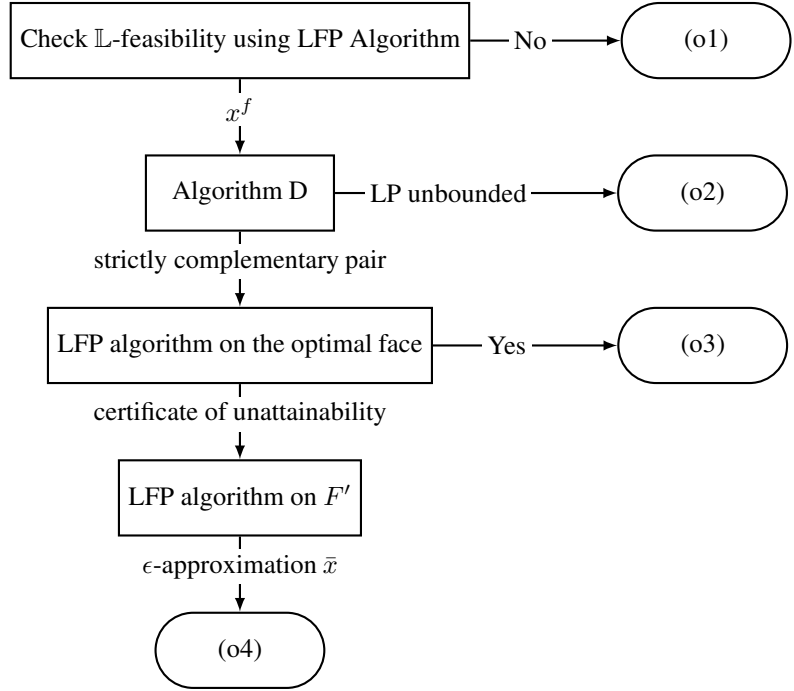

First we apply the LFP algorithm to $P=\{x:Ax\leq b\}$.
If it returns a certificate of real-infeasibility, or a certificate of $\cL$-infeasibility, we stop.
Otherwise, it returns a point $x^f\in P\cap\cL^n$.
We then consider the linear programming relaxation,
\begin{equation}\tag{P}
\max\{c^\top x:Ax\leq b\},
\end{equation}
of \eqref{dlp} and use Algorithm D to solve it.
Note that since \eqref{dlp} is feasible so is (P).
If (P) is unbounded, then Algorithm D will return a rational vector $r$ with $Ar\leq\0$.
After scaling we may assume that $r\in\Z^n$ and $x^f$, $r$ is a certificate of unboundedness.
Thus, we may assume that Algorithm D returns a strictly complementary pair $x^*$ of (P) and $y^*$ of its dual (D).
Then define the face 
\[
F:=Q\cap\{x:c^\top x\geq c^\top x^*\}.
\]
Let $I^=:=\supp(y^*)$ and let $A^=x\leq b^=$ be the inequalities from $Ax\leq b$ indexed by $I^=$.
Note that face $F$ is obtained from $P$ by setting some inequalities to equalities. 
Furthermore, by strict complementarity, $A^=x=b^=$ is the set of implicit equalities of $F$.
It follows that,
\begin{claim*}
$F=P\cap\{x:A^=x=b^=\}$ and $\aff(F)=\{x:A^=x=b^=\}$.
\end{claim*}
\noindent
We then apply the LFP algorithm to the system $Ax\leq b, A^=x\geq b^=$ that defines $F$.
If the algorithm returns a point $x^d\in F\cap\cL^n$ then $x^d$ is an optimal solution to \eqref{dlp} and we stop.
Otherwise the LFP algorithm will return a certificate $\bar{y},\bar{u}$ of $\cL$-infeasibility for $F$.
\Cref{infeasibility-certificate} implies that $\supp(\bar{y})\subseteq I^=$.
Then we have $\supp(\bar{u})\subseteq\supp(\bar{y})\subseteq I^==\supp(y^*)$.
It follows that $x^*,y^*,\bar{u}$ is a certificate of unattainability. 

At this juncture we look for an $\epsilon$-approximation by running the LFP algorithm for the polyhedron,
\[
F'=P\cap\{x:c^\top x\geq c^\top x^*-\epsilon\}.
\]
Note that since $A,b,c$ are rational, so is $c^\top x^*$, so $F'$ is a rational polyhedron.
By \Cref{epsilon} $F'$ contains a point in $\cL^n$.
We use the LFP algorithm once more to find such a point in $F'$.

\subsubsection{Solving linear programs versus solving $\cL$-linear programs}
Solving the $\cL$-linear program required at most two calls of the LFP procedure for outcomes (o1)-(o3) and three calls in case of outcome (o4). Thus the running time of our $\cL$-linear program solver is at most three times that of LFP. 

We have shown that we can use a black-box algorithm that solves the $\cL$-feasibility problem to solve $\cL$-linear programs.
Moreover, our LFP algorithm leverages a linear program solver to solve the $\cL$-feasibility problem.
Thus $\cL$-linear programs can be solved using a black-box linear program solver.
Here we show the converse, namely one may use a single call to a black-box algorithm for solving $\cL$-linear programs, 
to solve linear feasibility problems over rationals, i.e., finding a rational point in a polyhedron (which in turn is sufficient for solving linear programs).

Suppose we are given $A \in \Z^{m \times n}$ of full row rank and $b\in \Z^m$ and want to find $x \in \Q^n$ such that $Ax=b, x \geq \0$ or prove that no such $x$ exists.
We are given an $\cL$-linear program solver that works for a specific set $\cL$ where $\cL$ satisfies our (group), (density), and (membership) conditions.
Note that we do not get to choose $\cL$, it could be for instance that $\cL$ is the set of $p$-adic numbers for $p=2^{82,589,933}-1$ (which happens to be prime).
We feed the $\cL$-linear program solver the optimization problem:
\begin{equation}\label{BlackboxL}
\inf \left\{t \, : \, Ax + b t = b, \, x \geq \0, \, t \geq 0, \, x \in \cL^n, \, t \in \cL \right\},
\end{equation}
and we pick $\epsilon := 2^{-2 \mbox{ size}(A,b)}$.

Note that with $\bar{x} :=\0$, $\bar{t}:=1$, $(\bar{x},\bar{t}) \in \cL^{n+1}$ makes a feasible solution for \eqref{BlackboxL}.
Since \eqref{BlackboxL} is not unbounded ($t \geq 0$ for all feasible solutions), the only possible outcomes are (o3) and (o4). 
In case of (o3), the $\cL$-linear program solver returns a feasible solution $(\bar{x},\bar{t})$ of \eqref{BlackboxL} with the smallest possible $\bar{t}$. 
If $\bar{t} = 0$ then $\bar{x} \in \Q^n_{\geq 0}$ such that $A\bar{x}=b$. 
If $\bar{t}>0$ then \Cref{epsilon} proves that the LP relaxation of \eqref{BlackboxL} has optimal value $\bar{t}>0$, in particular $Ax=b, x\geq\0$ is infeasible.
In the case of outcome (o4), the solver returns a certificate of unattainability and an $\epsilon$-approximation $(\bar{x}, \bar{t})$. 
If $\bar{t} > \epsilon$ then again \Cref{epsilon} proves that the LP relaxation of \eqref{BlackboxL} has optimal value $\bar{t}>0$, and $Ax=b, x\geq\0$ is infeasible.
Otherwise, $\bar{t} \leq \epsilon$. 
Since $\bar{t} \leq 2^{-2 \mbox{ size}(A,b)}$, every extreme point of 
\[
\left\{ (x,t) \in \R^{n+1} \, : \, Ax +bt = b, \,\, x \geq \0, \,\, t \geq 0\right\}
\]
whose objective value is less than or equal to $\bar{t}$ must have $\bar{t} = 0$. 
Such an extreme point exists (and is guaranteed to be rational due to integrality of $A$ and $b$), 
and can be found in strongly polynomial time (i.e., the number of elementary arithmetic operations is bounded above by a polynomial function of $n$ only), starting with $(\bar{x}, \bar{t})$ (see, for instance, \cite{Megiddo1991} and \cite[Section 4.3]{Tuncel2010}). 
The ``$x$-part'' of such an extreme point is a rational solution of $Ax=b, x \geq \0$.

One of the complexity measures of our algorithms and analyses is based on the radius of the largest Euclidean ball contained in
some polyhedra. Similar complexity measures have been used before in analyzing the complexity of algorithms based on the ellipsoid method as well as interior-point algorithms for linear programs in particular, and convex optimization problems in general. We will remark on the primal-dual version. Consider the primal-dual pair of LPs where $A$ is $m$-by-$n$:
\[
\min\left\{c^{\top}x \, : \, Ax=b, \,\, x \geq \0\right\}
\hspace{.5cm} \textup{ and } \hspace{.5cm}
\max\left\{b^{\top}y \, : \, A^{\top} y + s =c, \,\, s \geq \0\right\}.
\]
Suppose that both are feasible. Let $[B,N]$ denote the strict complementarity partition
of $[n]$ for this primal-dual pair. One then defines
\begin{align*}
r_P(A,b,c) & :=  \min_{j \in B} \left\{\max\left\{ x_j \, : \, A_B x_B=b, \,\, x_B \geq \0\right\}\right\},\\
r_D(A,b,c) & := \min_{j \in N} \left\{\max\left\{ s_j \, : \,  A_B^{\top}y = c_B, \,\, A_N^{\top} y + s_N =c_N, \,\, s_N \geq \0\right\}\right\},\\
 r(A,b,c) & := \min\left\{r_P (A,b,c),r_D (A,b,c)\right\}.
 \end{align*}
 There are polynomial time algorithms for linear programs whose complexity are bounded above by a polynomial function of
 $n$ and $\log(1/r(A,b,c))$ \cite{MY1993,Ye1997} (assuming suitable feasible start, or one can apply the approach to a homogeneous self-dual
 reformulation, see \cite{YTM1994,HT2002}, the reformulation would change $r$ though). 
 Note that $r_P(A,b,c)$ is essentially the radius we used in our approach for problems in standard equality form. 
 $r_D(A,b,c)$ would also correspond to the radius we used for problems in inequality form, provided the columns of
 $A$ are scaled so that they are all approximately of unit norm. 
 In this latter case, $r(A,b,c)$ would also be relevant for our analysis when we are interested in finding solutions for primal-dual pair of $\cL$-linear programs.
\section{Bounding the fractionality of dyadic solutions} \label{section:fractionality}

Given a dyadic linear program $Ax\leq b, x \text{ dyadic}$ that is feasible, for some $A\in \Z^{m\times n}$ and $b\in \Z^m$, can we find a solution whose denominators are ``small"? In this section, we prove that if the program is feasible, then there exists a $\frac{1}{2^k}$-integral solution, where $k\leq \left\lceil\log_2 n+ (2n+1)\log_2(\|A\|_\infty\sqrt{n+1})\right\rceil$. In doing so, we take an integer linear programming perspective towards dyadic linear programs, in contrast to the linear programming approach taken in the previous sections. The results in this section are presented more generally for the $p$-adic numbers for any prime $p\geq 2$.

\paragraph{A hierarchy of integer linear programs.}
Let $p$ be a prime, and $\cL:=\left\{\frac{a}{p^\ell}:a,\ell\in \Z,\ell\geq 0\right\}$ the set of $p$-adic numbers. Consider the $\cL$-linear program:
\begin{equation}\label{eqn:plp}\tag{$\plp$}
\sup\{c^\top x:Ax\leq b, x \text{ is $p$-adic}\},
\end{equation} where $A\in \Z^{m\times n},b\in\Z^m,c\in\Z^n$. 
For each integer $k\geq 0$, consider the following restriction of \eqref{eqn:plp}: 
\begin{equation}\label{eqn:plpk}\tag{$\plp_k$}\max\{c^\top x:Ax\leq b,x \text{ is $1/p^k$-integral}\},
\end{equation} where we replaced $\sup$ by $\max$ because every element of the domain is isolated. In fact, \eqref{eqn:plpk} is equivalent to the ILP $\max\{c^\top z:Az\leq p^k b,z\in \Z^n\}$, in the sense that if $x$ is a feasible solution to \eqref{eqn:plpk} of value $\alpha$, then $p^k x$ is a feasible solution to the ILP of value $p^k \alpha$, and if $z$ is a feasible solution to the ILP of value $\beta$, then $\frac{1}{p^k}z$ is a feasible solution to \eqref{eqn:plpk} of value $\frac{1}{p^k} \beta$. 

Denote by $\sup(\plp)$ the supremum value of \eqref{eqn:plp}; if the $\cL$-LP is infeasible then $\sup(\plp):=-\infty$, and if it is unbounded then $\sup(\plp):=+\infty$.\footnote{
Observe that when $\sup(\plp)$ is finite, then $\sup(\plp) = \max\{c^\top x:Ax\leq b\}$ by \Cref{epsilon}. This equality, however, is not the focus in this section.}
Similarly, we define $\opt(\plp_k)$ for each integer $k\geq 0$. Clearly we have the following chain of inequalities: \begin{equation}\label{ineq-chain}
\opt(\plp_0)\leq\opt(\plp_1)\leq \opt(\plp_2)\leq \cdots \leq \sup(\plp) 
\end{equation} where any value in the chain may be finite or $\pm \infty$. 
In fact, we have the following theorem, thereby justifying our integer linear programming perspective.

\begin{theorem}
$\lim_{k\to \infty}\opt(\plp_k)=\sup(\plp)$. In fact, $\opt(\plp_k)=\sup(\plp)$ for sufficiently large $k$, unless $\sup(\plp)$ is finite and not attained by any feasible solution to \eqref{eqn:plp}. \end{theorem}
\begin{proof}
There are three cases: \begin{enumerate}
\item
If $\sup(\plp)=-\infty$, then clearly, $\opt(\plp_k)=\sup(\plp)$ for all $k\geq 0$.
\item
If $\sup(\plp)=+\infty$, then so is the linear relaxation of $(\plp)$, so there is an extreme ray $r\in \R^n$ of the polyhedron $\{x:Ax\leq b\}$ such that $c^\top r>0$. Given that $A,b$ have integral entries, we may assume that $r$ is integral. Let $\bar{x}$ be any feasible solution to $(\plp)$. Then $\bar{x}$ is $1/p^N$-integral for some integer $N\geq 0$. It can be readily checked that $(\bar{x}+t\cdot r:t=0,1,2,\ldots)$ is an infinite sequence of feasible solutions to $(\plp_k)$ of arbitrarily large objective value, for any integer $k\geq N$. Subsequently, $\opt(\plp_k)=+\infty$ for all $k\geq N$.
\item
Otherwise, $\sup(\plp)$ is finite. The inequalities in \eqref{ineq-chain}, together with the Monotone Convergent Theorem from Real Analysis, imply that $\lim_{k\to \infty} \opt(\plp_k)$ exists. In fact, the sequence must converge to $\sup(\plp)$ as we argue below.
 \begin{enumerate}
\item If the optimal value of $\eqref{eqn:plp}$ is attained, say by $\bar{x}$ that is $1/p^N$-integral for some integer $N\geq 0$, then $\bar{x}$ is an optimal solution to $(\plp_N)$, so $\opt(\plp_N)=\sup(\plp)$. Subsequently, it follows from \eqref{ineq-chain} that $\opt(\plp_k)=\sup(\plp)$ for all $k\geq N$.
\item Otherwise, $\sup(\plp)$ is finite but the optimal value is not attained by any feasible solution. In this case, by definition, there exists a sequence of feasible solutions $(\bar{x}^i:i\geq 1)$ to $(\plp)$ such that $\lim_{i\to \infty} c^\top \bar{x}^i = \sup(\plp)$. Observe that each $\bar{x}^i$ is feasible solution to $(\plp_{k})$ for some sufficiently large $k$. Subsequently, $$\sup(\plp)=\lim_{i\to \infty} c^\top \bar{x}^i\leq \lim_{k\to \infty}\opt(\plp_k)\leq \sup(\plp)$$ where the last inequality follows from \eqref{ineq-chain}. Equality must hold throughout, so we have $\lim_{k\to \infty} \opt(\plp_k)=\sup(\plp)$, as required.
\end{enumerate} 
\end{enumerate} 
\end{proof}

Given the theorem above, a natural question is for how small of a $k$ can we guarantee $\opt(\plp_k)=\sup(\plp)$? 
By switching to an equivalent model, for the sake of convenience, we may instead ask: given that $(\plp)$ is feasible, what is the smallest integer $k$ such that $(\plp_k)$ is feasible? 
In what follows, we provide a polynomial upper bound on the size of the smallest $k$.\footnote{Note that this statement has no bearing on the time complexity of solving $(\plp_k)$. Given a fixed integer $k\geq 0$ (possibly $k=0$), $(\plp_k)$ is NP-hard as its decision version is NP-complete (\emph{Given $A,b$ and $k\geq 0$, does $Ax\leq b$ have a $1/p^k$-integral solution?})~(see \cite{Arora09}, Chapter 2).} 

\begin{LE}\label{k-tilde}
Let $A$ be an $m\times n$ integer matrix of full row rank, and let $(B~\0)$ be the Hermite normal form of $A$, where $B$ is a square matrix. Then for all $b\in \Z^m$ and $k\in \Z_+$, the following statements are equivalent: \begin{enumerate}
\item[i.] $Ax=b$ has a $1/p^k$-integral solution,
\item[ii.] $B^{-1}b$ is $1/p^k$-integral.
\end{enumerate} Moreover, if $Ax=b$ has a $1/p^{\kappa}$-integral solution, and $\kappa\geq 0$ is the smallest such integer, then $p^{\kappa}$ divides $\gcd(A)$.
\end{LE}
\begin{proof}
Let $U$ be a unimodular matrix such that $AU=(B~\0)$ is in Hermite normal form, where $B$ is a square matrix. Let $I,J$ be the sets of column labels of $B,\0$ in $(B~\0)$. Observe that $\{x:Ax=b\}=\{Uz : z_I=B^{-1}b, z_J \text{ free}\}$. 
Since $U$ is unimodular, $U^{-1}$ is integral, so $Uz$ is $1/p^k$-integral if and only if $z=U^{-1}(Uz)$ is $1/p^k$-integral. This implies that $Ax=b$ has a $1/p^k$-integral solution if, and only if, $B^{-1}b$ is $1/p^k$-integral, so (i) and (ii) are equivalent.

Suppose now that $Ax=b$ has a $1/p^{\kappa}$-integral solution, and $\kappa$ is the smallest such integer. If $\kappa=0$, then $p^{\kappa}=1$, so we are done. Otherwise, $\kappa\geq 1$. It follows from our choice of $\kappa$ that $B^{-1}b$ is not $1/p^{\kappa-1}$-integral but it is $1/p^{\kappa}$-integral. On the other hand, by Cramer's rule, $B^{-1}b$ is $1/\det(B)$-integral, so $p^{\kappa} \mid \det(B)$. Since $\det(B)=\gcd(A)$, it follows that $p^{\kappa}$ divides $\gcd(A)$, as required. (For more on $\gcd(A)$, see \Cref{sec:primer-lattices}.)
\end{proof}

Given the system $Ax=b$ above, we can use the algorithm of Kannan and Bachem~\cite{Kannan79} to compute the Hermite normal form in strongly polynomial time. The state-of-the-art can be found in Storjohann's PhD thesis, giving an algorithm with running time complexity $nm^\omega$, where $\omega\in (2,2.376)$ is the matrix multiplication exponent, and the absolute value of each entry of $B,U$ is bounded above by $m (\sqrt{m}\|A\|_\infty)^{2m}$~(\cite{storjohannPhD2000}, Chapter 6). 
Either of these leads to a strongly polynomial time algorithm for finding $\kappa$, or certifying that it does not exist. 

Moving forward, for $A\in \Z^{m\times n}$, define 
\begin{align*}
\xi_p(A,b)&:=\min
\{k\in \Z_+: Ax\leq b \text{ has a $1/p^k$-integral solution}\} \qquad \forall~b\in \Z^m\\
\xi_p(A)&:= \max\left\{\xi_p(A,b):b\in \Z^m \text{ s.t. } Ax\leq b \text{ has a $p$-adic solution}\right\}.
\end{align*}
Thus, for any $b\in \Z^m$ such that $Ax\leq b$ has a $p$-adic solution, there is a $1/p^k$-integral solution for $k=\xi_p(A)$.
In what follows, we provide an upper bound on $\xi_p(A)$ that depends only polynomially on $n$ and the encoding size of $A$.

\begin{theorem}\label{xi-upper-bound}
$\xi_p(A)\leq \left\lceil\log_p n+ (2n+1)\log_p(\|A\|_\infty\sqrt{n+1})\right\rceil$ for all $A\in \Z^{m\times n}$. 
\end{theorem}
\begin{proof}
Let $P:=\{x:Ax\leq b\}$. We need to prove that for some $k\leq \mathrm{RHS}$, $P$ contains a $1/p^{k}$-integral point. We shall apply \Cref{find-xip}.

Denote by $A^= x\leq b^=$ the subsystem comprised of all implicit equalities of $Ax\leq b$, and by $A^< x \leq b^<$ the subsystem for the other inequalities. Let $F:=\{x:A^= x=b^=\}$, which is the affine hull of $P$. 

Recall that $B_{\infty}(x,\epsilon)\subseteq \R^n$ is the $\infty$-norm closed ball of radius $\epsilon$ centered at $x$. 
Let $\tilde{\epsilon}:=\max\big\{\epsilon:\exists~x\in P \text{ s.t. } B_{\infty}(x,\epsilon)\cap F\subseteq P\big\}$, and let $\tilde{B}$ be an $\infty$-norm closed ball of radius $\tilde{\epsilon}$ such that $\tilde{B}\cap F\subseteq P$.

Finally, let $\kappa:=\min\{k:F \text{ contains a $1/p^{k}$-integral point}\}$. Let $z$ be a $1/p^{\kappa}$-integral point in $F$. Let $\ell:=\mathrm{nullity}(A^=)$, and let $d^1,\ldots,d^\ell\in \Z^n$ be a basis for $\ker(A^=)$. Note that $F = z + \spn\{d^1,\ldots,d^\ell\}$. 

\begin{claim} 
$P$ contains a $1/p^k$-integral point for all $k\geq \max\left\{\kappa,\left\lceil\log_{p} \left(\frac{\ell\max\{\|d^1\|_{\infty},\ldots,\|d^\ell\|_{\infty}\}}{\tilde{\epsilon}}\right)\right\rceil\right\}$.
\end{claim}
\begin{cproof}
Since $k\geq \log_{p} \left(\frac{\ell\max\{\|d^1\|_{\infty},\ldots,\|d^\ell\|_{\infty}\}}{\tilde{\epsilon}}\right)$, there exists a $1/p^k$-integral $\bar{\rho}\in \mathbb{R}^n$ such $z+\bar{\rho}\in P$, by \Cref{find-xip}. Moreover, since $k\geq \kappa$ and $z$ is $1/p^{\kappa}$-integral, it follows that $z+\bar{\rho}$ is $1/p^k$-integral; $z+\bar{\rho}$ is the desired point (which in fact belongs to $\tilde{B}$).
\end{cproof}

It therefore suffices to upper bound $\kappa$, $\log_p(\max_{i}\|d^i\|_{\infty})$ and $-\log_p(\tilde{\epsilon})$.

\begin{claim} 
$p^{\kappa}$ divides $\gcd(A^=)$. In particular, $\kappa\leq (n-\ell)\log_{p}(\|A^=\|_{\infty}\sqrt{n-\ell})$.
\end{claim}
\begin{cproof}
Note that $\rank{A^=}=n-\ell$.
Let $A'x=b'$ be a system of $n-\ell$ linearly independent constraints of $A^=x=b^=$. Observe that $\{x:A'x=b'\}=\{x:A^= x=b^=\}$. It follows from \Cref{k-tilde} that $p^{\kappa}$ divides $\gcd(A')$ which in turn divides every subdeterminant of $A'$ of order $n-\ell$. By repeating this argument for every choice of $A'$, we obtain that $p^{\kappa}$ divides every subdeterminant of $A^=$ of order $n-\ell$, so $p^{\kappa}$ divides $\gcd(A^=)$. 
By Hadamard's inequality (see \Cref{sec:primer-lattices} for more), $\gcd(A^=)\leq (\|A^=\|_{\infty}\sqrt{n-\ell})^{n-\ell}$, so $\kappa\leq (n-\ell)\log_{p}(\|A^=\|_{\infty}\sqrt{n-\ell})$.
\end{cproof}

\begin{claim} 
We may choose $d^1,\ldots,d^\ell$ such that $\log_p(\max_i\|d^i\|_{\infty})\leq (n-\ell)\log_{p}(\|A^=\|_{\infty}\sqrt{n-\ell})$.
\end{claim}
\begin{cproof}
Let $A'$ be a row submatrix of $A^=$ comprised of $n-\ell$ linearly independent rows. Observe that $A',A^=$ have the same kernel. Let $A''$ be a square nonsingular column submatrix of $A'$. Denote by $I$ and $J$ the sets of column labels of $A'$ inside and outside $A''$, respectively. Note that $|J|=\ell$. For each $j\in J$, the $j\textsuperscript{th}$ column of $A'$ can be expressed as a unique linear combination of the columns of $A''$. Subsequently one obtains $\ell$ vectors in $\ker{A'}$, say $c^j\in \mathbb{R}^n, j\in J$, where each $c^j$ has only one nonzero entry in $J$, namely in the $j\textsuperscript{th}$ position. In particular, $c^j,j\in J$ are linearly independent and therefore form a basis of $\ker(A')$. By applying Cramer's rule, we see that for all $j\in J,i\in I$, we have $c^j_i = \pm \frac{\det(B)}{\det(A'')}$ where $B$ is obtained from $A''$ by swapping out the $i\textsuperscript{th}$ column for the $j\textsuperscript{th}$ column of $A'$. The desired vectors $d^1,\ldots,d^\ell$ may be picked as $\det(A'')\cdot c^j,j\in J$. Note that $\max_i\|d^i\|_{\infty}$ is at most the absolute value of the maximum subdeterminant of $A'$ of order $n-\ell$. Thus, by Hadamard's inequality, $\max_i\|d^i\|_{\infty}\leq (\|A'\|_{\infty}\sqrt{n-\ell})^{n-\ell}$, thereby proving the claim.
\end{cproof}

\begin{claim} 
$\tilde{\epsilon}\geq \frac{1}{|M| \max_{i} (|A^<|\1)_i}$ where $M$ is a minor of $\left(\begin{smallmatrix} A^=&\0\\A^<&\1\end{smallmatrix}\right)$, and $|A^<|$ is obtained from $A^<$ by replacing every entry with its absolute value. In particular, $-\log_p(\tilde{\epsilon}) \leq (n+1)\log_p(\|A\|_{\infty}\sqrt{n+1})+\log_p(n\|A^<\|_{\infty})$.
\end{claim}
\begin{cproof}
Consider the LP $\max\{t:A^=x=b^=, A^< x+t\1\leq b^<\}$. Let $(\tilde{x},\tilde{t})$ be a basic optimal solution to the LP. Our choice of $A^<x\leq b^<$ immediately implies that $\tilde{t}>0$, and so because $A,b$ have integral entries, it follows from Cramer's rule that $\tilde{t}\geq \frac{1}{|M|}$ where $M$ is a minor of $\left(\begin{smallmatrix} A^=&\0\\A^<&\1\end{smallmatrix}\right)$. Let $\epsilon:=\frac{\tilde{t}}{\max_i (|A^<| \1)_i}$. 
We claim that $B_{\infty}(\tilde{x},\epsilon)\cap F\subseteq P$. To this end, pick a point $x\in F$ such that $\|x-\tilde{x}\|_\infty\leq \epsilon$. Then $$
A^< (x-\tilde{x})\leq |A^<(x-\tilde{x})|\leq |A^<||x-\tilde{x}|\leq \epsilon |A^<|\1\leq \tilde{t}\1\leq b^<-A^< \tilde{x},$$
so $A^<x\leq b^<$, implying that $x\in P$. Thus, $B_{\infty}(\tilde{x},\epsilon)\cap F\subseteq P$. By definition, we must have $\tilde{\epsilon}\geq \epsilon$, and since $\epsilon\geq \frac{1}{|M| \max_{i} (|A^<|\1)_i}$, we get that $\tilde{\epsilon}\geq \frac{1}{|M| \max_{i} (|A^<|\1)_i}$. The inequality $-\log_p(\tilde{\epsilon}) \leq (n+1)\log_p(\|A\|_{\infty}\sqrt{n+1})+\log_p(n\|A^<\|_{\infty})$ follows by an application of Hadamard's inequality.
\end{cproof}

Observe that our upper bounds on $\kappa$ and $\log_p(\max_i\|d^i\|_{\infty})$ are matching, thus we may ignore $\kappa$ when applying Claim~1. Putting the claims together, we see that for some integer $k\geq 0$ satisfying
$$k\leq 
\left\lceil\log_p \ell+\log_p(\max_i\|d^i\|_{\infty})-\log_p(\tilde{\epsilon})\right\rceil\leq
\left\lceil\log_p n+ (2n+1)\log_p(\|A\|_\infty\sqrt{n+1})\right\rceil,
$$
$P$ contains a $1/p^{k}$-integral point, thereby proving the theorem.
\end{proof}

\begin{CO}
Fix a matrix $A\in \Z^{m\times n}$. Suppose $p$ is a sufficiently large prime number. Then the following statements hold: \begin{enumerate}
\item[a.] For any $b\in \Z^m$, if $Ax=b$ has a $p$-adic solution, then it has an integral solution.
\item[b.] For any $b\in \Z^m$, if $Ax\leq b$ has a $p$-adic solution, then it has a $1/p$-integral solution. That is, $\lim_{p\to \infty}\xi_p(A)\leq 1$. \end{enumerate}
\end{CO}
\begin{proof}
{\bf (a)} Suppose $p$ is a prime number such that $p>\gcd(A)$, and that for some $b\in \Z^m$, $Ax=b$ has a $p$-adic solution. Let $\kappa$ be the smallest integer $k\geq 0$ such that $Ax=b$ has a $1/p^{\kappa}$-integral solution. It then follows from \Cref{k-tilde} that $p^{\kappa}\mid \gcd(A)$. Since $p>\gcd(A)$, it follows that $\kappa=0$. Thus, $Ax=b$ has an integral solution.

{\bf (b)} 
We know from \Cref{xi-upper-bound} that $\xi_p(A)\leq
\left\lceil\log_p n+ (2n+1)\log_p(\|A\|_\infty\sqrt{n+1})\right\rceil$. Given a fixed $A$, the RHS approaches $1$ as $p$ tends to $\infty$, implying in turn that $\lim_{p\to \infty}\xi_p(A)\leq 1$, as required.
\end{proof}

Another consequence of \Cref{xi-upper-bound} is the following, bringing this section to an end.

\begin{CO}
Let $A\in \Z^{m\times n}$. Given a feasible dyadic linear program $Ax\leq b, x \text{ dyadic}$, for some $b\in \Z^m$, there exists a $\frac{1}{2^k}$-integral solution where $k\leq \left\lceil\log_2 n+ (2n+1)\log_2(\|A\|_\infty\sqrt{n+1})\right\rceil$.\qed
\end{CO}
\section{Bounding the support size of dyadic solutions} \label{section:support}

Take $A\in \mathbb{Z}^{m\times n}$ and $b\in \mathbb{Z}^m$. 
Given a dyadic linear program of the form $Ax=b,x\geq \0 \text{ and dyadic}$, that is feasible, does there exist a solution with a small number of nonzero entries? We have placed dyadic linear program as a problem on the spectrum between LP and ILP, so let us give a brief overview of the state-of-the-art for the two extremes.

For the case of real solutions to $Ax=b,x\geq \0$, we have the upper bound of $m$ on the support size of a solution, given by Carath\'{e}odory's famous theorem. In fact, the same guarantee holds for an optimal solution to the standard equality form LP $\min\{w^\top x:Ax=b,x\geq \0\}$, for any $w\in \mathbb{R}^n$ for which there is a finite optimum.

For the case of integral solutions to $Ax=b,x\geq \0$, a ``Carath\'{e}odory-type" upper bound of $2m\log_2(4m\|A\|_\infty)$ on the size of the support was first established in \cite{Eisenbrand06} by making an elegant use of the \emph{Pigeonhole Principle}. This bound was later improved to $2m\log_2(2\sqrt{m}\|A\|_\infty)$ in \cite{Aliev17} by the use of \emph{Siegel's Lemma}, a tool from the Geometry of Numbers, which we shall later see and prove in this section. In fact, the latter also obtained the same guarantee for an optimal solution to $\min\{w^\top x: Ax=b, x\geq \0 \text{ and integral}\}$, for any $w\in \mathbb{R}^n$ for which there is a finite optimum.

A natural first step for studying both extremes, as well as dyadic linear program, is to obtain guarantees for the system of linear equations $Ax=b$, with the appropriate restriction on the domain of $x$. Observe that the guarantees above carry over in a black-box fashion to this setting, by simply transforming $Ax=b$ to $[A~-A]\big(\begin{smallmatrix}y\\ z\end{smallmatrix}\big) = b, \big(\begin{smallmatrix}y\\ z\end{smallmatrix}\big)\geq \0$ with the substitution formula $x=y-z$.

With this context in mind, let us now move on to the spectrum between LP and ILP. In the previous section, for the sake of transparency but also generality, we provided guarantees not just for dyadic linear programs but more generally $p$-adic linear programs, for any prime $p\geq 2$. In this section, for the same reasons, we will provide guarantees for a different extension of dyadic linear programs, as explained below.

For every integer $k\geq 1$, denote by $p_k$ the $k\textsuperscript{th}$ prime number. Recall from \Cref{section:foundation} that a rational number is $[p_k]$-adic if it is of the form $\frac{a}{b}$ where $a\in \Z$ and $b$ is a product of primes in $[p_k]$. Note that a rational number is $[2]$-adic if and only if it is dyadic. Let $p_0:=1$. We also deal with \emph{$[p_0]$-adic numbers}, which are precisely the integers.\footnote{Note that the $[p_0]$-adic numbers are discrete as opposed to the $[p_k]$-adic numbers for $k\geq 1$ which form a dense set. However, discreteness will be irrelevant in this section.} The guarantees we provide in this section will apply more generally to optimal solutions to the following $[p_k]$-adic linear program, for any integer $k\geq 0$: \begin{align}
\min\{w^\top x: Ax&=b,x\geq \0 \text{ and $[p_k]$-adic}\}. \label{2-sparsity-LS-2}
\end{align} Observe that for $k=0$ we recover integer linear programming, for $k=1$ dyadic linear programming, and for $k=\infty$ linear programming. Our bounds are given exactly but indirectly in terms of proxy functions, and also loosely but directly as a function of $p_k$, $m$, and the maximum absolute value $\|A\|_{\infty}$ of an entry of $A$; none of the bounds however depend on $n$.
Along the way, we shall also provide improved and tight guarantees for solutions to 
\begin{align}
Ax&=b, x \text{ $[p_k]$-adic.} \label{2-sparsity-LS-1}
\end{align}

\subsection{Examples where every solution has full support}

Let us present instances of \eqref{2-sparsity-LS-2} and \eqref{2-sparsity-LS-1} where every solution has full support. The example below is inspired by \cite{Cook86}.

\begin{EG}\label{full-support-eg-1}
Let $q_1,\ldots,q_{n}$ be distinct primes such that $q_i\geq p_{k+1}$ for all $i\in [n]$. Let $Q:=q_1q_2\cdots q_{n}$ and $A:=\left(\frac{Q}{q_1}~\frac{Q}{q_2}~\cdots~\frac{Q}{q_{n}}\right) \in \mathbb{R}^{1\times n}$. Then the system $Ax = 1, x \text{ $[p_k]$-adic}$ is feasible, and every solution has full support.

One can extend this example to one with an arbitrary number $m$ of rows, by replacing $Ax=1$ by $(A\otimes I) y = \1$, where $\otimes$ denotes the Kronecker product, $I$ the $m$-by-$m$ identity matrix, and $\1$ the all-ones $m$-dimensional column vector.
\end{EG}
\begin{proof}
We only prove the first statement, and leave the easy verification of the second statement to the reader. Since the entries of $A$ have GCD $1$, it follows from B\'{e}zout's Lemma that $Ax=1$ has an integral, thus $[p_k]$-adic solution. This proves feasibility of the system $Ax=1, x \text{ $[p_k]$-adic}$. Now let $\bar{x}$ be a feasible solution. Suppose for a contradiction that $\bar{x}$ does not have full support, say $\bar{x}_n=0$. For each $i\in [n-1]$, write $\bar{x}_i=\frac{a_i}{b_i}$ where $a_i,b_i\in \Z$ and $b_i$ is a product of primes in $[p_k]$. Let $B$ be the largest common multiple of $b_i,i\in [n-1]$, which is also a product of primes in $[p_k]$. Then we have the identity $$
\sum_{i=1}^{n-1} A_{1i} \frac{a_iB}{b_i} = B.
$$ By construction, the GCD of $A_{1i},i\in [n-1]$ is $q_n$.
Since each $\frac{a_iB}{b_i},i\in [n-1]$ is an integer, it therefore follows from the identity above that $q_n\mid B$, a contradiction since $q_n$ is a prime outside $[p_k]$ while $B$ is a product of primes in $[p_k]$.
\end{proof}

This example can naively be extended to the inequality case.

\begin{EG}\label{full-support-eg-1.5}
Let $A$ be the matrix from \Cref{full-support-eg-1}. Take a resigning $A^{s}$ of $A$, obtained by negating some entries of $A$, such that $A^sx = 1, x\geq \0$ has an integral solution. It can be readily checked that $A^sx = 1, x\geq \0, x \text{ $[p_k]$-adic}$, is feasible, and every solution has full support.
Once again, this example can be extended to one with an arbitrary number $m$ of rows: $(A^s\otimes I) y = \1, y\geq \0, y \text{ $[p_k]$-adic}$.
\end{EG}

For $k=0$, there is another interesting example, which also appears in \cite{Eisenbrand06}.

\begin{EG}\label{full-support-eg-2}
Let $A:=(1~2~2^2~\cdots~2^{n-1})
 \in \mathbb{R}^{1\times n}$. Consider the integer linear program $$\min\left\{\1^\top x: Ax=2^n-1, x\geq \0 \text{ and integral}\right\}.$$ Then $x=\1$ is the unique optimal solution, which in particular has full support.
As before, this example can readily be extended to one with an arbitrary number $m$ of rows: $$\min\left\{\1^\top y: (A\otimes I)y=(2^n-1)\cdot \1, y\geq \0 \text{ and integral}\right\}.$$
\end{EG}
\begin{proof}
We only verify the first statement, and leave the proof of the second statement to the reader. Clearly, $x^\star=\1$ is a feasible solution to integer linear program. Take another feasible solution $\bar{x}$. Let $i\in [n]$ be the largest index such that $\bar{x}_i>1$. Since $\bar{x}$ is integral and nonnegative, and $A\bar{x}=2^n-1$, it follows that $i<n$. We now change $\bar{x}$ by updating $\bar{x}_i:=\bar{x}_i-2$ and $\bar{x}_{i+1}:=\bar{x}_{i+1}+1$; note that this change reduces $\1^\top \bar{x}$ by $1$. By repeatedly applying this operation, we obtain the solution $x^\star=\1$ to the integer linear program. This procedure proves in particular that $x^\star=\1$ is the unique optimal solution to the integer linear program.
\end{proof}

\subsection{From full support solutions to large prime factors}

In this section, we pave the way for obtaining an upper bound on the support size for \eqref{2-sparsity-LS-2}, and an even better bound for \eqref{2-sparsity-LS-1}. The gap between these two bounds is due to the following lemma.

\begin{LE}\label{full-support-2}
Let $A\in \mathbb{Z}^{m\times n},b\in \mathbb{Z}^m,w\in \mathbb{R}^n$, and take an integer $k\geq 0$. Then the following statements hold: \begin{enumerate}
\item Suppose $Ax=b, x \text{ $[p_k]$-adic}$, is feasible, and every solution has full support. Then for every integral solution $\bar{x}$ to $Ax=\0$, every nonzero entry has a prime factor greater than or equal to~$p_{k+1}$.
\item Suppose $\min\{w^\top x: Ax=b, x\geq \0, x \text{ $[p_k]$-adic}\}$ has an optimal solution, and every optimal solution has full support. Then for every nonzero integral solution $\bar{x}$ to $Ax=\0$, there exists some nonzero entry with a prime factor greater than or equal to $p_{k+1}$.
\end{enumerate}
\end{LE}
\begin{proof}
{\bf (1)}
Let $\bar{x}$ be a nonzero integral vector such that $A\bar{x}=\0$. 
Suppose for a contradiction that for some index $i$, $\bar{x}_i\neq 0$ and every prime factor of $\bar{x}_i$ is at most $p_k$. Let $x^1$ be a $[p_k]$-adic solution to $Ax=b$. By the hypothesis, $x^1$ has full support, so $x^1_i\neq 0$. Let $x^2:=x^1 - \frac{x^1_i}{\bar{x}_i}\bar{x}$, which is another solution to $Ax=b$. Since every prime factor of $\bar{x}_i$ belongs to $\{p_1,\ldots,p_k\}$, it follows that $x^2$ is another $[p_k]$-adic solution to $Ax=b$, one whose support excludes $i$, a contradiction to our hypothesis.

{\bf (2)} is similar to (1), except that in order to ensure $x^2$ remains nonnegative given the nonnegativity of $x^1$, we would need the prime factors of every nonzero entry of $\bar{x}$ to be less than $p_{k+1}$. Let us elaborate. Let $\bar{x}$ be a nonzero integral vector such that $A\bar{x}=\0$.

First, we prove that $w^\top \bar{x}=0$. Suppose otherwise. Let $x^\star$ be an optimal solution to $\min\{w^\top x: Ax=b, x\geq \0, x \text{ is $[p_k]$-adic}\}$. By assumption, $x^\star$ has full support, implying in turn that for a sufficiently small and $[p_k]$-adic $\epsilon>0$, both $x^\star\pm \epsilon \bar{x}$ are feasible solutions to the $[p_k]$-adic linear program. However, since $w^\top \bar{x}\neq 0$, one of $x^\star\pm \epsilon \bar{x}$ would have a strictly smaller objective value than $x^\star$, thereby contradicting the optimality of $x^\star$. Thus, $w^\top \bar{x}=0$.

Secondly, we prove that some nonzero entry of $\bar{x}$ has a prime factor greater than or equal to $p_{k+1}$. Suppose otherwise. We shall use $\bar{x}$ to construct an optimal solution to $\min\{w^\top x: Ax=b, x\geq \0, x \text{ is $[p_k]$-adic}\}$ without full support, thereby contradicting the hypothesis. To this end, let $x^1$ be an optimal solution to the $[p_k]$-adic linear program. By the hypothesis, $x^1$ has full support. Choose $$i\in\arg\min_{j\in [n]}\left\{\frac{x^1_j}{\bar{x}_j}:\bar{x}_j\neq 0\right\}.$$ Let $x^2:=x^1 - \frac{x^1_i}{\bar{x}_i}\bar{x}$. Our choice of $i$ ensures that $x^2$ is defined and nonnegative. Since every prime factor of $\bar{x}_i$ belongs to $\{p_1,\ldots,p_k\}$ by the contrary assumption, $x^2$ is $[p_k]$-adic. Thus, since $w^\top \bar{x}=0$, it follows that $x^2$ is another optimal solution to $\min\{w^\top x: Ax=b, x\geq \0, x \text{ is $[p_k]$-adic}\}$, one whose support excludes $i$, a contradiction to our hypothesis.
\end{proof}

\subsection{Large prime factors: examples and analysis}

It may not be clear how to find examples that satisfies the conclusions of \Cref{full-support-2} parts~(1) and~(2). Let us present two examples, both of which we will prove to be extremal in a sense. Throughout the subsection we assume that $k\in \Z_{\geq 0}$.

\begin{EG}\label{large-primes-eg-1}
Let $q_1,\ldots,q_{n}$ be distinct primes such that $q_i\geq p_{k+1}$ for all $i\in [n]$. Let $Q:=q_1q_2\cdots q_{n}$ and $A:=\left(\frac{Q}{q_1}~\frac{Q}{q_2}~\cdots~\frac{Q}{q_{n}}\right) \in \mathbb{R}^{1\times n}$. It can be readily checked that for any integral solution $\bar{x}$ to $A x=\0$, and for each $i\in [n]$, we have $q_i \mid \bar{x}_i$, so if $\bar{x}_i\neq 0$ then $\bar{x}_i$ has a prime factor greater than or equal to $p_{k+1}$. Thus the $1$-by-$n$ matrix $A$ satisfies the conclusion of \Cref{full-support-2}~(1).
\end{EG}

First, notice that the example above comes from \Cref{full-support-eg-1}, indicating that its essence is captured by \Cref{full-support-2}~(1). Secondly, note that in the example above, the entries of the row vector $A$ are ``large" with respect to the GCD of its entries; this normalization is necessary since scaling $A$ does not change its kernel. More precisely, the size of every entry divided by the GCD of the entries, is at least $p_{k+1}p_{k+2}\cdots p_{k+n-1}$. In \Cref{det-lower-bound} below, we prove this bound more generally for a full-row-rank matrix with $m$ rows, where the notions of the ``size of an entry" and the ``GCD of the entries" have been replaced by ``a nonzero order-$m$ minor" and the ``GCD of order-$m$ minors". The argument is inspired by a similar one in~\cite{Dubey23+}.

\begin{LE}\label{det-lower-bound}
Let $A$ be an $m$-by-$n$ integral matrix of full row rank, where for every integral solution $\bar{x}$ to $Ax=\0$, if $\bar{x}_i\neq 0$ then $x_i$ has a prime factor greater than or equal to $p_{k+1}$. Then for every subset $I\subseteq [n]$ of size $m$ such that $\det(A_{[m]\times I})\neq 0$ we have
$$
\frac{|\det(A_{[m]\times I})|}{\gcd(A)}\geq \left\{
\begin{array}{ll}
p_{k+1}^{n-m} &\text{ if $n< 2m$}\\
p_{k+1}^m p_{k+2}^m \cdots p_{k+\lfloor \frac{n}{m}\rfloor -1}^m p_{k+\lfloor \frac{n}{m}\rfloor}^{n-m\lfloor \frac{n}{m}\rfloor} 
&\text{ otherwise.}
\end{array}
\right.
$$ Moreover, there is a subset $J\subseteq [n]$ of size $m$ such that $\det(A_{[m]\times J})\neq 0$, and
$$
\frac{|\det(A_{[m]\times J})|}{\gcd(A)}\geq \left\{
\begin{array}{ll}
p_{k+2}^{n-m} &\text{ if $n< 2m$}\\
p_{k+2}^m p_{k+3}^m \cdots p_{k+\lfloor \frac{n}{m}\rfloor}^m p_{k+\lfloor \frac{n}{m}\rfloor+1}^{n-m\lfloor \frac{n}{m}\rfloor} 
&\text{ otherwise.}
\end{array}
\right.
$$
\end{LE}
\begin{proof}
Let $U$ be a unimodular matrix such that $AU=(B~\0)$, where $B$ is an $m\times m$ matrix. (For instance, $U$ can be the unimodular matrix that brings $A$ into Hermite normal form after some elementary unimodular column operations.) We know that $\gcd(A) = |\det(B)|$. Let $(U_1~U_2)$ be the partition of $U$ into two column submatrices such that $AU_1=B$ and $AU_2=\0$. Observe that the $n-m$ columns of $U_2$ form a basis for $\ker(A)$. In fact, since $U$ is a unimodular matrix, the columns of $U_2$ form a basis for the lattice $L:=\ker(A)\cap \mathbb{Z}^n$, that is, $\{U_2y:y\in \mathbb{R}^{n-m}\}\cap \mathbb{Z}^n=\{U_2 y:y\in \mathbb{Z}^{n-m}\}$. Given that $L=\overline{L}$, it follows that $\gcd(U_2)=[\overline{L}:L]=1$ (see \Cref{sec:primer-lattices} for more).

Given a prime number $p$, we say that \emph{$p$ divides a row of $U_2$} if it divides every entry of the row.

\begin{claim} 
Every row of $U_2$ is divisible by a prime greater than or equal to $p_{k+1}$.
\end{claim}
\begin{cproof}
Let $g\in \mathbb{Z}_{\geq 1}$ be the GCD of the entries of a nonzero row of $U_2$, say row $i$. Then there exists a $y\in \mathbb{Z}^{n-m}$ such that $(U_2 y)_i = g$. Since $x:=U_2 y\in L$, our key assumption implies that every nonzero entry of $x$, and in particular $x_i=g$, has a prime factor greater than or equal to $p_{k+1}$, as claimed.
\end{cproof}

\begin{claim} 
Every prime number $p$ divides at most $m$ rows of $U_2$.
\end{claim}
\begin{cproof}
Suppose otherwise, that is, for some prime number $p$, there is a subset $J\subseteq [n]$ of row indices of $U_2$ such that each row of $U_2$ with an index in $J$ is divisible by $p$, and $|J|\geq m+1$. Every $(n-m)$-by-$(n-m)$ submatrix of $U_2$ contains a row index from $J$, therefore its determinant is divisible by $p$. This implies that $p \mid \gcd(U_2)$, which is a contradiction since $\gcd(U_2)=1$.
\end{cproof}

Denote by $I_1,I_2$ the set of column labels of $U_1,U_2$, respectively. Then $|I_1|=m$ and $|I_2|=n-m$.

\begin{claim} 
For every subset $I\subseteq [n]$ of size $m$ such that $\det(U_{\overline{I}\times I_2})\neq 0$, we have
$$
|\det(U_{\overline{I}\times I_2})|\geq 
\left\{
\begin{array}{ll}
p_{k+1}^{n-m} &\text{ if $n< 2m$}\\
p_{k+1}^m p_{k+2}^m \cdots p_{k+\lfloor \frac{n}{m}\rfloor -1}^m p_{k+\lfloor \frac{n}{m}\rfloor}^{n-m\lfloor \frac{n}{m}\rfloor} 
&\text{ otherwise.}
\end{array}
\right.
$$ Moreover, there is a subset $J\subseteq [n]$ of size $m$ such that $\det(U_{\overline{J}\times I_2})\neq 0$, and
$$
|\det(U_{\overline{J}\times I_2})|\geq 
\left\{
\begin{array}{ll}
p_{k+2}^{n-m} &\text{ if $n< 2m$}\\
p_{k+2}^m p_{k+3}^m \cdots p_{k+\lfloor \frac{n}{m}\rfloor}^m p_{k+\lfloor \frac{n}{m}\rfloor+1}^{n-m\lfloor \frac{n}{m}\rfloor} 
&\text{ otherwise.}
\end{array}
\right.
$$
\end{claim}
\begin{cproof}
Take a subset $I\subseteq [n]$ of size $m$ such that $\det(U_{\overline{I}\times I_2})\neq 0$.
By Claim~1, every row of $U_{\overline{I}\times I_2}$ is divisible by a prime in $P:=\{p_{k+1},p_{k+2},\ldots\}$. On the other hand, by Claim~2, every prime in $P$ is divisible by at most $
\min\{|\overline{I}|,m\}=
\min\{n-m,m\}$ rows of $U_{\overline{I}\times I_2}$. These two facts immediately imply the first inequality. To get the second, stronger inequality, it suffices to choose $J:=I$ such that it includes the indices of all the rows of $U_2$ divisible by $p_{k+1}$ (of which there are at most $m$ by Claim~2), and $\det(U_{\overline{J}\times I_2})\neq 0$. This can be done by using the fact that $\gcd(U_2)=1$, so there exists a minor of $U_2$ of order $n-m$ that is not divisible by $p_{k+1}$; this will be precisely $\det(U_{\overline{J}\times I_2})$.
\end{cproof}

\begin{claim} 
For every subset $I\subseteq [n]$ of size $m$, we have $|\det(A_{[m]\times I})| = \gcd(A) \cdot |\det(U_{\overline{I}\times I_2})|$.
\end{claim}
\begin{cproof}
It follows from $A=(B~\0)U^{-1}$ that $A_{[m]\times I} = 
(B~\0) (U^{-1})_{[n]\times I}=
B\cdot (U^{-1})_{I_1\times I}$. Subsequently, $$|\det(A_{[m]\times I})| = |\det(B)|\cdot  |\det((U^{-1})_{I_1\times I})| = \gcd(A)\cdot |\det((U^{-1})_{I_1\times I})|.$$ The right-hand side term can be rewritten in terms of the adjugate $\mathrm{adj}(U)$ of $U$. Observe that $\mathrm{adj}(U) = \det(U)\cdot U^{-1}$, so $U^{-1}=\pm \mathrm{adj}(U)$ because $U$ is unimodular. Thus, 
$$
|\det((U^{-1})_{I_1\times I})| = |\det(\mathrm{adj}(U)_{I_1\times I})|= |\det(U_{\bar{I}\times I_2})|\cdot |\det(U)|^{m-1}= |\det(U_{\bar{I}\times I_2}) |
$$ where the second equality follows from Jacobi's Theorem on the adjugate matrix (see \Cref{sec:primer-lattices} for more), and the last equality from the unimodularity of $U$. Combining the two lines of equalities above yields the claim.
\end{cproof}

The lemma readily follows from Claims~3 and~4.
\end{proof}

Next, let us present an example that satisfies the conclusion of \Cref{full-support-2}~(2).

\begin{EG}\label{large-primes-eg-2}
Let $A:=\left(1~p_{k+1}~p_{k+1}^2~\cdots~p_{k+1}^{n-1}\right) \in \mathbb{R}^{1\times n}$. It can be readily checked that for any nonzero integral solution $\bar{x}$ to $A x=\0$, the smallest index $i\in [n]$ such that $\bar{x}_i\neq 0$ satisfies $p_{k+1} \mid \bar{x}_i$. Thus the $1$-by-$n$ matrix $A$ satisfies the conclusion of \Cref{full-support-2}~(2).
\end{EG}

First, notice that this example is the same as the feasible region of \Cref{full-support-eg-2} for $k=0$, indicating that in this case the essence of the example is captured by \Cref{full-support-2}~(2). Secondly, in the example above, the row vector $A$ has small entries as well as large, however its $2$-norm is ``large" relative to the GCD of its entries. More precisely, the $2$-norm of the row vector divided by the GCD of the entries, is equal to $\sqrt{\frac{p_{k+1}^{2n}-1}{p^2_{k+1}-1}}$; the latter is sandwiched between $p_{k+1}^{n-1}$ and $\sqrt{1+\frac{1}{p^2_{k+1}-1}}\cdot p_{k+1}^{n-1}$. In \Cref{siegel-LE} below, known as \emph{Siegel's Lemma}, we prove the lower bound more generally for a full-row-rank matrix with $m$ rows, where the ``$2$-norm of the row vector" is replaced by ``the $m$-dimensional volume of the parallelepiped generated by the rows of $A$", and as before, the ``GCD of the entries" is replaced by the ``GCD of order-$m$ minors". 

\begin{LE}[Siegel's Lemma, see~\cite{Bombieri83}]\label{siegel-LE}
Consider a linear system $Ax=\0$, where $A\in \mathbb{Z}^{m\times n}$ has full row rank, and let $\ell\geq 1$. Suppose for every nonzero integral solution $\bar{x}$ to $Ax=\0$, $\|\bar{x}\|_\infty\geq \ell$. Then $\sqrt{\det(AA^\top)}/\gcd(A)\geq \ell^{n-m}$.
\end{LE}
\begin{proof}
Consider the lattice $L:=\{A^\top y:y\in \mathbb{Z}^m\}$ and its orthogonal complement lattice $L^\perp:=\left\{x\in \Z^n:A x=\0\right\}$. By assumption, the convex set $Q:=\{x:\|x\|_{\infty}< \ell\}\cap \{x:A x=\0\}$, which is symmetric about the origin, contains no nonzero vector of the lattice~$L^\perp$. We may therefore apply Minkowski's First Theorem to upper bound the $(n-m)$-dimensional volume of $Q$ (see \Cref{sec:primer-lattices} for more). The theorem implies that $\vol_{n-m}(Q) \leq 2^{n-m} \det(L^\perp)$. 
On the one hand,  $\frac{1}{2\ell}\cdot Q$ is an $(n-m)$-dimensional affine slice of the unit hypercube $\{x:\|x\|_\infty<\frac12\}$ going through the origin, so by Vaaler~\cite{Vaaler79}, $\vol_{n-m}(\frac{1}{2\ell}Q)\geq 1$, implying in turn that $\vol_{n-m}(Q)\geq (2\ell)^{n-m}$.
On the other hand, given that $\overline{L}:=\{A^\top y:y\in \R^m\}\cap \Z^n\supseteq L$, we have
$$\det(L^\perp) = \det(\overline{L})= \frac{\det(L)}{[\overline{L}:L]}=\frac{\det(L)}{\gcd(A^{\top})}=\frac{\sqrt{\det(A A^\top)}}{\gcd(A)}$$ (see \Cref{sec:primer-lattices} for more). Putting everything together we obtain the desired inequality.
\end{proof}

\subsection{Proxy and direct upper bounds on the support size}

We obtain the following exact upper bound on the support size of solutions to $[p_k]$-adic linear systems and $[p_k]$-adic linear programs. Note that the upper bounds are provided indirectly by suitable proxy functions of the input size.

\begin{theorem}\label{sparsity-2}
Let $A\in \mathbb{Z}^{m\times n}$, $b\in \mathbb{Z}^m$, and $w\in \mathbb{R}^n$. Then the following statements hold for every integer $k\geq 0$:
\begin{enumerate}
\item If $Ax=b,x \text{ $[p_k]$-adic}$, is feasible, then it has a solution with support size at most $n'$, where 
for some full-row-rank $m'$-by-$n'$ submatrix $A'$ of $A$, we have 
$$
\Delta_{m'}(A')\geq 
\left\{
\begin{array}{ll}
p_{k+2}^{\frac{n'}{m'}-1} &\text{ if $n'< 2m'$}\\
p_{k+2} p_{k+3} \cdots p_{k+\lfloor \frac{n'}{m'}\rfloor} p_{k+\lfloor \frac{n'}{m'}\rfloor+1}^{\frac{n'}{m'}-\lfloor \frac{n'}{m'}\rfloor} 
&\text{ otherwise,}
\end{array}
\right.
$$ where $\Delta_{m'}(A')$ is the $m'\textsuperscript{th}$ root of the maximum absolute value of an order-$m'$ minor of $A'$.
\item If $\min\{w^\top x:Ax=b,x\geq \0,x \text{ $[p_k]$-adic}\}$ has an optimal solution, then it has an optimal solution with support size at most $n'$, where for some full-row-rank submatrix $A'$ of $A$ with $m'$ rows, $n'$ satisfies $$\left(\frac{\sqrt{\det(A'A'^\top)}}{\gcd(A')}\right)^{\frac{1}{m'}}\geq p_{k+1}^{\frac{n'}{m'}-1}.$$ 
\end{enumerate}
\end{theorem}
\begin{proof}
{\bf (1)} Let $x^\star$ be a solution to $Ax=b, x \text{ $[p_k]$-adic}$, with minimum support size. After moving to a submatrix of $A$, and the corresponding subvector of $b$, if necessary, we may assume that $x^\star$ has full support, and $A'=A$ has full row rank; thus $m'=m\leq n=n'$. 
By \Cref{full-support-2}~(1), for every integral solution $\bar{x}$ to $Ax=\0$, every nonzero entry of $\bar{x}$ has a prime factor greater than or equal to $p_{k+1}$. Thus, by \Cref{det-lower-bound}, $A$ has an $m$-by-$m$ submatrix $B$ such that $$|\det(B)|\geq 
\left\{
\begin{array}{ll}
p_{k+2}^{n-m} &\text{ if $n< 2m$}\\
p_{k+2}^m p_{k+3}^m \cdots p_{k+\lfloor \frac{n}{m}\rfloor}^m p_{k+\lfloor \frac{n}{m}\rfloor+1}^{n-m\lfloor \frac{n}{m}\rfloor} 
&\text{ otherwise.}
\end{array}
\right.
$$ 
Taking the $m\textsuperscript{th}$ of both sides, and taking advantage of the inequality $\Delta_m(A)\geq |\det(B)|^{1/m}$ by definition, we obtain (1).

{\bf (2)}
Let $x^\star$ be an optimal solution to $\min\{w^\top x:Ax=b,x\geq \0, x \text{ $[p_k]$-adic}\}$ with minimum support size. After moving to a submatrix of $A$, and the corresponding subvectors of $w$ and $b$, if necessary, we may assume that $x^\star$ has full support, and $A'=A$ has full row rank; thus $m'=m\leq n=n'$. By \Cref{full-support-2}~(2), for every integral solution $\bar{x}$ to $Ax=\0$, there exists a nonzero entry with a prime factor greater than or equal to $p_{k+1}$, so in particular $\|\bar{x}\|_{\infty}\geq p_{k+1}$.
It therefore follows from Siegel's Lemma that $$\frac{\sqrt{\det(AA^\top)}}{\gcd(A)}\geq p_{k+1}^{n-m},$$ proving (2). 
\end{proof}

Observe that for fixed $A$, both bounds above guarantee that $\frac{n'}{m'}\to 1$ as $p_k\to \infty$, thereby matching Carath\'{e}odory's bounds for the support size of solutions to linear systems, and optimal solutions to linear programs.

We can use the theorem above, along with Hadamard's inequality, to provide loose but direct upper bounds on the support size of solutions. 

\begin{theorem}\label{sparsity-2-asymptotics}
Let $A\in \mathbb{Z}^{m\times n}$, $b\in \mathbb{Z}^m$, and $w\in \mathbb{R}^n$. Then the following statements hold for every integer $k\geq 0$:
\begin{enumerate}
\item If $Ax=b,x \text{ $[p_k]$-adic}$, is feasible, then it has a solution with support size at most $n'$, where for $r=\lfloor\frac{n'}{m}\rfloor$, the following holds: if $r\geq 1+2e$ and $\ln(\sqrt{m}\|A\|_\infty)\geq e$, then
$$
r\leq 1+\frac{2(1+e)}{e}\frac{\ln(\sqrt{m}\|A\|_\infty)}{\ln\ln(\sqrt{m}\|A\|_\infty)}.
$$
\item If $\min\{w^\top x:Ax=b,x\geq \0,x \text{ $[p_k]$-adic}\}$ has an optimal solution, then it has an optimal solution with support size at most $n'$, where for $r=\frac{n'}{m}$, we have
$$
r\leq \left\{
\begin{array}{lll}
1+ \ln{(m\|A\|^2_\infty)}/(2\ln{p_{k+1}}-1)
 &\text{for all $k\geq 0$,}\\ \\
1+\log_{2}\|A\|_\infty
+\log_{2} (m\|A\|_\infty)\cdot \frac{1+\log_{2}\|A\|_\infty}{1+2\log_{2}\|A\|_\infty}  \qquad&\text{ if $k=0$ and $n'\geq 4\|A\|^2_\infty$.}
\end{array}
\right.
$$
\end{enumerate}
\end{theorem}
\begin{proof}
(1) 
For some full-row-rank $m'$-by-$n'$ submatrix $A'$ of $A$, 
the inequality of \Cref{sparsity-2}~(1) holds. We may assume that $A'=A$, $m'=m$ and $n'=n$. Let $r:= \lfloor \frac{n}{m}\rfloor$. 
If $r=1$ then there is nothing to prove. Otherwise, $r\geq 2$. Then by \Cref{sparsity-2}~(1) we have $$\Delta_m(A)\geq \prod_{i=1}^{r-1} (p_{k+2}+i-1)\geq \left(p_{k+2}+\left\lceil\frac{ r-2}{2}\right\rceil\right)^{\left\lceil\frac{ r-1}{2}\right\rceil}.$$ On the other hand, by Hadamard's inequality, $\sqrt{m}\|A\|_\infty\geq \Delta_m(A).$ 
Combining this inequality with the one above, and then taking the natural logarithm, we obtain $$\left\lceil\frac{r-1}{2}\right\rceil\ln\left(p_{k+2}+\left\lceil\frac{r-2}{2}\right\rceil\right) \leq \ln(\sqrt{m}\|A\|_\infty).$$ Let $x:=\frac{r-1}{2}$ and $y:=\ln(\sqrt{m}\|A\|_\infty)$. Then the inequality above implies that $x\ln(x)\leq y$.

\begin{claim} 
For $x,y\geq e$, we have
$x\leq \frac{1+e}{e}\cdot\frac{y}{\ln(y)}$.
\end{claim}
\begin{cproof}
Let $z:=x\ln(x)$. Then for $x\geq 1$, $$
\frac{z}{\ln(z)}= \frac{x\ln(x)}{\ln(x)+\ln\ln(x)}\geq \frac{x\ln(x)}{(1+1/e)\ln(x)}=\frac{e}{1+e}x
$$ where the middle inequality holds because $\ln\ln(x)\leq \frac{1}{e}\ln(x)$ for all $x\geq e$. Since the function $\frac{2t}{\ln(t)}$ is increasing on $t\in [e,\infty)$, and $y\geq z\geq e$, it follows that $\frac{1+e}{e}\cdot\frac{y}{\ln(y)}\geq \frac{1+e}{e}\cdot\frac{z}{\ln(z)}$, so the claim follows.~\end{cproof} 

Claim~1 proves the inequality of (1).

(2) 
For some full-row-rank $m'$-by-$n'$ submatrix $A'$ of $A$, 
the inequality of \Cref{sparsity-2}~(2) holds.
We may assume that $A'=A$, $m'=m$ and $n'=n$. Thus, we have that $$\sqrt{\det(AA^\top)}\geq p_{k+1}^{n-m}.$$
Suppose $a^1,\ldots,a^m\in \Z^n$ are the rows of $A$. By Hadamard's inequality,
$$
\sqrt{\det(AA^\top)}\leq \sqrt{\prod_{i=1}^m \|Aa^i\|_2}
\leq 
\sqrt{\prod_{i=1}^m (\sqrt{m}\sqrt{n}\|A\|^2_\infty)}
\leq 
\left(\sqrt{n}\|A\|_{\infty}\right)^{m}.
$$ 
Combining the inequalities above, and then taking logarithms base $p_{k+1}$, we obtain \begin{equation}\label{sparsity2-eq1}
n\leq m+m\log_{p_{k+1}}\left(\sqrt{n}\|A\|_\infty \right).
\end{equation} 
Let $f,g:\mathbb{R}_{\geq 1}\to \mathbb{R}_{\geq 0}$
be the functions defined as $g(x)=x+x\log_{p_{k+1}}\|A\|_\infty$ and $f(x):=\log_{p_{k+1}}x$. 
Then \eqref{sparsity2-eq1} may be rewritten as $n\leq g(m)+\frac{m}{2} f(n)$. Observe that $f$ is an increasing function.

\begin{claim} 
The following inequalities hold: \begin{enumerate}
\item[a.] $f(g(m)+\frac{m}{2} f(n))\leq f(g(m)) + \frac{m}{2\ln(p_{k+1})g(m)} f(n)$,
\item[b.] for $k=0$, if $n\geq 4\|A\|^2_\infty$, then $f(g(m)+\frac{m}{2} f(n))\leq f(g(m)) + \frac{m}{2g(m)} f(n)$.
\end{enumerate}
\end{claim}
\begin{cproof}
To see the inequalities, note that $\log_{p_{k+1}}(x+y) = \log_{p_{k+1}}(x)+\log_{p_{k+1}}(1+\frac{y}{x})\leq \log_{p_{k+1}}(x)+\frac{1}{\ln(p_{k+1})}\frac{y}{x}$ for all $x,y$ over which the LHS and RHS are defined.
Thus, $$
f\left(g(m)+\frac{m}{2} f(n)\right) = \log_{p_{k+1}}\left(g(m)+\frac{m}{2} f(n)\right)\leq \log_{p_{k+1}}g(m)+\frac{m\cdot f(n)}{2\ln(p_{k+1})g(m)} 
$$ thereby proving (a). When $k=0$ (i.e.\ $p_{k+1}=2$), we can replace $\log_{p_{k+1}}(1+\frac{y}{x})\leq \frac{1}{\ln(p_{k+1})}\frac{y}{x}$ by the improved inequality $\log_{p_{k+1}}(1+\frac{y}{x})\leq \frac{y}{x}$ as long as $\frac{y}{x}\geq 1$. Thus, after repeating the above argument with this improved inequality, we obtain (b).
\end{cproof}

Subsequently, \begin{align*}
n\leq g(m)+\frac{m}{2} f(n)
&\leq g(m)+\frac{m}{2}f\left(g(m)+\frac{m}{2} f(n)\right)\quad\text{since $f$ is increasing}\\ \\
&\leq g(m)+\frac{m}{2}f(g(m)) + \frac{m}{2}\frac{m}{2\ln(p_{k+1})g(m)} f(n) \quad\text{by part (a) of Claim~2}\\ 
&\quad\vdots\\ 
&\leq g(m)+\frac{m}{2}f(g(m)) \sum_{t=0}^{\infty} \left(\frac{m}{2\ln(p_{k+1})g(m)}\right)^t
\\ \\
&= g(m)+ \frac{m}{2}f(g(m))\cdot \frac{2\ln(p_{k+1})g(m)}{2\ln(p_{k+1})g(m)-m}.
\end{align*} Substituting for $g(m)$ and $f(g(m))$, and dividing both sides by $m$, we get the following inequality: 
\begin{align*}
\frac{n}{m}&\leq 
1+\log_{p_{k+1}}\|A\|_\infty
+\frac12\log_{p_{k+1}}(m+m\log_{p_{k+1}}\|A\|_\infty)\cdot \frac{1+\log_{p_{k+1}}\|A\|_\infty}{1+\log_{p_{k+1}}\|A\|_\infty-\frac{1}{2\ln(p_{k+1})}}
\\ \\
&\leq 
1+\log_{p_{k+1}}\|A\|_\infty
+\left(\log_{p_{k+1}} m+\frac{\log_{p_{k+1}}\|A\|_\infty}{\ln(p_{k+1})}\right)\cdot \frac{1+\log_{p_{k+1}}\|A\|_\infty}{2+2\log_{p_{k+1}}\|A\|_\infty - \frac{1}{\ln(p_{k+1})}}\\ \\
&\leq 
1+\log_{p_{k+1}}\|A\|_\infty
+\left(\log_{p_{k+1}} m+\frac{\log_{p_{k+1}}\|A\|_\infty}{\ln(p_{k+1})}\right)\cdot \frac{1}{2 - \frac{1}{\ln(p_{k+1})}}\\ \\
&= 1+ \frac{\ln{m}+2\ln{\|A\|_\infty}}{2\ln{p_{k+1}}-1}.
\end{align*} 
If $k=0$ and $n\geq 4\|A\|^2_\infty$, then we can use part (b) instead of part (a) of Claim~2 in the inequalities above, and obtain
$$n\leq g(m)+ \frac{m}{2}f(g(m))\cdot \frac{2g(m)}{2g(m)-m}.$$ Substituting for $g(m)$ and $f(g(m))$, and dividing both sides by $m$, we get the following inequality: 
\begin{align*}
\frac{n}{m}&\leq 
1+\log_{2}\|A\|_\infty
+\frac12\log_{2}(m+m\log_{2}\|A\|_\infty)\cdot \frac{1+\log_{2}\|A\|_\infty}{1+2\log_{2}\|A\|_\infty-\frac{1}{2}}
\\ \\
&\leq 
1+\log_{2}\|A\|_\infty
+\left(\log_{2} m+\log_{2}\|A\|_\infty\right)\cdot \frac{1+\log_{2}\|A\|_\infty}{1+2\log_{2}\|A\|_\infty},
\end{align*} as required.
\end{proof}

Let us mention a few notable cases. First,
the upper bound of \Cref{sparsity-2-asymptotics}~(1) for $k=0$ is interesting in its own right, and was recently obtained in~\cite{Dubey23+} (their upper bound is given in terms of big $O$ notation, in contrast to our bound). Secondly, for $k=0$, thanks to an improved asymptotic analysis, the upper bound of \Cref{sparsity-2-asymptotics}~(2) gives a minor improvement over the guarantee of $2m\log_2(2\sqrt{m}\|A\|_\infty)$ in \cite{Aliev17}.
Thirdly, for $k=0$ and $\|A\|_\infty=1$, the upper bound of \Cref{sparsity-2-asymptotics}~(2) simplifies to the following.

\begin{CO}
Let $A\in \{0,\pm 1\}^{m\times n},b\in \Z^m$.
If $\min\{w^\top x:Ax=b,x\geq \0, x\in \Z^n\}$ has an optimal solution, then it has one with support size at most $m(1+\log_2 m)\approx m(1+1.45\ln{m})$.\qed
\end{CO}

Finally, for dyadic linear programs, we obtain the following, bringing this subsection to an end.

\begin{CO}\label{sparsity-dyadic-CO}
Let $A\in \Z^{m\times n},b\in \Z^m$ and $w\in \R^n$. 
If $\min\{w^\top x:Ax=b,x\geq \0,x \text{ dyadic}\}$ has an optimal solution, then it has one with support size at most $
m(1+\ln(m\|A\|^2_\infty)/(2\ln{3}-1))\approx
m(1+0.84\ln{m}+1.68\ln\|A\|_\infty)$.
\qed
\end{CO}

\subsection{Examples revisited}\label{subsec:examples-revisited}

\begin{customEG}{\ref{full-support-eg-1} revisited}
Consider the example $(A\otimes I)y = \1, y \text{ $[p_k]$-adic}$ with $mn$ variables and $m$ equations. Let $q_i:=p_{k+i}$ for $i\in [n]$.
Observe that $$\Delta_m(A\otimes I)
= \left(\frac{Q}{q_1}\right)^m= (p_{k+2}p_{k+3}\cdots p_{k+n})^m.
$$ Observe that the RHS is precisely the lower bound provided by \Cref{sparsity-2}~(1). This shows that (a) this example is extremal, and (b) the lower bound given on $\Delta_m$ by \Cref{sparsity-2}~(1) cannot be improved in the case when the number of variables is a multiple of $m$.
\end{customEG}

\begin{customEG}{\ref{full-support-eg-1.5} revisited}
Consider the example 
$\min\{0:(A^s\otimes I) y = \1, y\geq \0, y \text{ $[p_k]$-adic}\}$ with $mn$ variables and $m$ equations. Let $q_i:=p_{k+i}$ for $i\in [n]$.
Let $B:=A^s\otimes I$ and consider the inequality of \Cref{sparsity-2}~(2). The LHS is $$
\left(\frac{\sqrt{\det(BB^\top)}}{\gcd(B)}\right)^{1/m} = \sqrt{\sum_{i=1}^n \left(\frac{Q}{q_i}\right)^2} = \Theta\left(p_{k+2}\cdots p_{k+n}\sqrt{n}\right)
$$ while the RHS is $
p_{k+1}^{n-1}$, so there is a multiplicative gap of $$
\Theta\left(\frac{p_{k+2}}{p_{k+1}}\cdots \frac{p_{k+n}}{p_{k+1}}\sqrt{n}\right)
$$
between the LHS and RHS for this example.
\end{customEG}

\begin{customEG}{\ref{full-support-eg-2} revisited} For this example we have $k=0$. Consider the integer linear program 
$$\min\left\{\1^\top y: (A\otimes I)y=(2^n-1)\cdot \1, y\geq \0 \text{ and integral}\right\}.$$
Let $B:=A\otimes I$ and consider the inequality of \Cref{sparsity-2}~(2). 
The LHS is $$
\left(\frac{\sqrt{\det(BB^\top)}}{\gcd(B)}\right)^{1/m} =\sqrt{\sum_{i=1}^n \left(2^{i-1}\right)^2}\in \left[2^{n-1}, \sqrt{\frac43} \cdot 2^{n-1}\right)
$$ while the RHS is $2^{n-1}$, so there is a constant multiplicative gap of $\sqrt{4/3}$ between the LHS and RHS for this example. This shows that (a) this example is extremal up to the constant factor, and (b) the lower bound given in \Cref{sparsity-2}~(2) cannot be improved beyond the constant factor, in the case when $k=0$.
\end{customEG}
\section{Concluding remarks and future research}\label{section:conclusion}

In this paper, we studied dyadic linear programming and its extension to $\cL$-linear programming, where $\cL$ is a dense subset of $\R$ closed under addition and negation. Two important extensions were $p$-adic and $[p]$-adic linear programs, for a prime $p$. 

We laid the foundation for $\cL$-linear programming by characterizing feasibility, stating optimality conditions, classifying all the possible outcomes and providing concise certificates in each case. A distinguishing feature was that unlike linear programs, an $\cL$-linear program may have an optimal value that is converged to but never attained within the feasible region.

We proved that under mild assumptions on $\cL$, namely that $\cL$ comes with a membership oracle and contains all $p$-adic numbers for some explicitly given prime $p$, an $\cL$-linear program can be solved in polynomial time. In fact, we established a constant factor equivalence between the running times of solving an $\cL$-linear program and  a linear program, with blackbox reductions going in either direction. 

Going beyond the blackbox reductions, if we are given deeper access to an algorithm for solving LPs which is guaranteed to find strictly complementary solutions (when the instance has an optimal solution), we can modify such an algorithm in part by inserting our subroutines in suitable places so that the LP algorithm is run only once (with these inserted subroutines) to solve the corresponding $\cL$-linear program. Such modifications can be particularly straightforward for the two-phase algorithms for LPs (as one can follow the analysis in \Cref{sec-optimization} to see how to modify the LP algorithm).

Our blackbox approach can also be useful in solving \emph{$\cL$-convex programs}. For a special class of  
convex optimization problems at hand, if we are able to compute affine hulls (of the feasible region and the optimal face) and
obtain rational representations for them (when possible) then, we can solve the $\cL$-convex program using the approach in \Cref{sec-optimization}. We do not need the convex programming instance to satisfy strict complementarity, but we would require the convex optimization algorithm to compute rational vectors (when they exist)  in the relative interiors of the corresponding sets. Limits of sequences of solutions generated by many interior-point algorithms for convex optimization lie in the relative interiors of the corresponding sets.

An irony of our polynomial algorithm, and even some of the foundational results, is that even though we set to solve an $\cL$-linear program, the numbers encountered throughout may in fact fall outside $\cL$. For instance, in the theorem of the alternatives, \Cref{alternative}, the non-existence certificate $u$ must inevitably be outside $\cL$. That said, the algorithm does successfully characterize and solve the various outcomes of an $\cL$-linear program, and provides the first step for finding a polynomial algorithm where all the numbers involved in the computations belong to $\cL$.

Given $Ax\leq b, x \text{ $p$-adic}$, that is feasible, what is the smallest $k\in \Z_{\geq 0}$ such that there is a $1/p^k$-integral solution? While determining $k$ is NP-hard, we provided upper bounds on $k$ that are polynomial in $n$ and the encoding size of $A$. A particular case of interest comes from combinatorial optimization. Given a graph $G=(V,E)$ and a nonempty set $T\subseteq V$ of even size, a \emph{$T$-join} is an edge set whose odd-degree vertices coincides with $T$. 
It is known that the \emph{fractional $T$-join packing} problem
$$
\max\left\{\1^\top y :
\sum\left(y_J:J\ni e\right) \leq 1 ~\forall e\in E;
y_J\geq 0~\forall \text{ $T$-joins $J$}\right\}$$ has an optimal solution that is dyadic, i.e., $\frac{1}{2^k}$-integral for some integer $k\geq 0$~\cite{Abdi-Tjoins}. The proof provides no upper bound guarantee on $k$. That said, it has been conjectured by Seymour that $k\leq 2$ (\cite{Cornuejols01}, Conjecture 2.15, also see Schrijver \cite{Schrijver03} 79.3e). An upper bound of $k\leq c\log(|E|)$ for some universal constant $c$, also remains open. The conjecture of Seymour combined with our approach in the current manuscript, suggests a study of classes of linear programs with integral data such that for every integral objective function vector, the primal has a $\frac{1}{2^{k_1}}$-integral optimal solution (whenever it has an optimal solution) and the dual has a $\frac{1}{2^{k_2}}$-integral optimal solution, for some fixed pair of nonegative integers $k_1$ and $k_2$. Seymour's Conjecture above for ideal clutters corresponds to the special case $k_1:=0$, $k_2:=2$.

Given a $[p]$-adic linear program $\min\{w^\top x:Ax=b,x\geq \0, x \text{ $[p]$-adic}\}$ that has an optimal solution, where $A$ has $m$ rows, we provided upper bound guarantees on the support size of an optimal solution, where the bound depended polynomially on $m,p$ and the encoding size of $A$. A helpful twist in this case was extending the notion of $[p]$-adic numbers to include the case of $p=1$, by declaring the $[1]$-adic numbers as the integers. As such, we obtained a spectrum of guarantees ranging from ILPs ($p=1$) on the one end, passing through dyadic linear programs ($p=2$), and reaching LPs ($p=\infty$) on the other and matching Carath\'{e}odory's bound.
Along the way, we also provided tight upper bounds on the support size of a solution to a feasible $[p]$-adic linear system of the form $Ax=b, x \text{ $[p]$-adic}$.

While our upper bound guarantees for $[p]$-adic linear programs are tight for the two ends of the spectrum, $p=1$ (\Cref{full-support-eg-2}) and $p=\infty$, there remains a gap between our best lower bound (\Cref{full-support-eg-1.5}) and our upper bound for $2\leq p<\infty$, as discussed in \Cref{subsec:examples-revisited}. We believe that due to the density of the feasible region for $p\geq 2$, the upper bounds in this case should look more like the upper bounds for $[p]$-adic linear systems.

In the special case of $\|A\|_\infty=1$, the best lower bound on the support size of an optimal solution to a $[p]$-adic linear program that we can show is at most $O(m)$, while our upper bound is $O(m\ln m)$. Closing the gap in this case remains an intriguing open question. An important special case comes in the dyadic ($p=2$) case from the fractional $T$-join packing problem mentioned above. By developing a column generation technique for solving dyadic linear programs, and by leveraging tools from matching theory, we achieve a matching upper bound of $O(m)$ (note $m=|E|$ in this case)~\cite{DLP-sequel}.

We have used the size of the input, in particular $\ln\|A\|_{\infty}$ to state our results (for bounds on computational complexity as well as support size bounds etc.). However, for specially structured instances, there are better complexity measures, capturing more intrinsic properties of the instance. This typically yields tighter and more insightful bounds. Thus, it would be fruitful to pursue this direction in future research.

Let $a^1,\ldots,a^n\in \Z^m$. The set $\{a^1,\ldots,a^n\}$ is a \emph{dyadic generating set for a cone (DGSC)} if every integral vector in the conic hull of the vectors can be expressed as a dyadic conic combination of the vectors. This notion was coined and studied in our first work on dyadic linear programming~\cite{acgt23}. 
Given a DGSC $\{a^1,\ldots,a^n\}$ and an integral vector $b$ in the conic hull, we know that $b$ can be expressed as a dyadic conic combination of the vectors. What is the fewest number $k$ of nonzero coefficients in such a representation? While \Cref{sparsity-dyadic-CO} gives an upper bound of $O(m\ln(m\|A\|^2_{\infty}))$ on $k$, we conjecture that there is a $O(m)$ upper bound on $k$. The rationale behind this comes from the observation that a DGSC may be viewed as the dyadic analogue of \emph{Hilbert bases} for integer linear programming~\cite{Giles79} for which the analogous upper bound guarantee is $2m-2$~\cite{Sebo90}.

Finally, we propose a weakening of Seymour's dyadic conjecture.

\begin{CN}
Let $A\in \{0,1\}^{m\times n}$ be an ideal matrix, and for some $c\in \Z^n_{\geq 0}$, let $\tau_c:=\min\{c^\top x:Ax\geq \1,x\geq \0\}$. Let $p$ be the largest prime in $[\tau_c]$. Then $\max\{\1^\top y:A^\top y\leq c,y\geq \0\}$ has a $[p]$-adic optimal solution. 
\end{CN}

This conjecture has been verified for the clutter of dijoins of a digraph (\cite{Hwang22}, Theorem 2.13).



\section*{Acknowledgements}

We would like to acknowledge pleasurable discussions with Yatharth Dubey and Siyue Liu about sparsity aspects of this work; progress towards \Cref{det-lower-bound} was most fruitful. This work was supported in part by ONR grant N00014-22-1-2528, EPSRC grant EP/X030989/1 and Discovery Grants from NSERC. 

{\small \bibliographystyle{abbrv}\bibliography{references}}

\appendix

\section{Primer on lattices and determinants}\label{sec:primer-lattices}

Generally we follow \cite{Martinet03,Beckenbach61,Prasolov94} as reference texts on lattices, inequalities, and equalities in linear algebra, respectively. More specifically, we use the following specific results in the paper.

\paragraph{Hermite normal form.} 
Let $A\in \Z^{m\times n}$ be a matrix of full row rank.
Then $A$ can be brought into \emph{Hermite normal form} by means of \emph{elementary unimodular column operations}. In particular, there exists an $n$-by-$n$ unimodular matrix $U$ such that $AU=(B~\0)$, where $B$ is a non-singular $m$-by-$m$ matrix, and $\0$ is an $m$-by-$(n-m)$ matrix with zero entries. See (\cite{ConfortiCornuejolsZambelli2014}, Section 1.5.2) or (\cite{Schrijver98}, Chapter 4) for more details.

\paragraph{Matrix GCD.} Denote by $\gcd(A)$ the greatest common divisor (GCD) of all order-$m$ minors of~$A$. This quantity can be found as follows. First, bring $A$ into Hermite normal form after applying elementary unimodular column operations, say we obtain $(B~\0)$ where $B$ is a square matrix. Then $\gcd(A)=\gcd(B)=|\det(B)|$, where the first equality can be readily checked by the reader, and the second one follows immediately from definition.

\paragraph{Lattice ``closure" and index.}
Let $L$ be an $m$-dimensional lattice contained in $\Z^n$, and let $\overline{L}:=\aff(L)\cap \Z^n$.
Clearly, $\overline{L}$ is a lattice containing $L$. Observe that within the affine hull of $L$, and confined to the integers, $\overline{L}$ is the maximal lattice containing $L$. It is known that $\overline{L}$ can be partitioned into a number of integral linear translations of $L$; this number is denoted as the \emph{index} $[\overline{L}:L]\in \Z_{\geq 1}$. Thus, $L=\overline{L}$ if and only if $[\overline{L}:L]=1$.

\paragraph{Lattice determinant and orthogonal complement.}
The \emph{determinant of $L$}, denoted $\det(L)$, is the $m$-dimensional volume of its fundamental parallelepiped. The inverse of $\det(L)$ may be used to measure the ``density" of $L$: the smaller the determinant, the larger the density. It is known that $\det(L) = [\overline{L}:L]\det(\overline{L})$~(see \cite{Martinet03}, Proposition 1.1.5~(4)). The \emph{orthogonal complement} of $L$ is the lattice $L^\perp:=\left\{y\in \Z^n: y^\top x=0~~\forall~x\in L\right\}$. It can be readily checked that $L^\perp$ has index $1$, that is, $L^\perp=\overline{L^\perp}$. It is known that $\det(L^\perp) = \det(\overline{L})$ (see \cite{Martinet03}, Proposition 1.9.8).

\paragraph{Lattice formulas.} Let $C\in \Z^{n\times m}$ be a matrix with $m$ linearly independent columns, and consider the lattice $L:=\left\{C y:y\in \Z^m\right\}\subseteq \Z^n$. Let us find $\det(L),\overline{L},[\overline{L}:L]$ and $L^\perp$ in terms of~$C$. First, the fundamental parallelepiped of $L$ is $\left\{C \lambda:\0\leq \lambda<\1\right\}$, whose $m$-dimensional volume is $\det(L)$. Subsequently, it can be readily checked that $\det(L)^2$ is the Gram determinant of the columns of $C$, that is, $\det(L)=\sqrt{\det(C^\top C)}$. 
It can be readily seen that $\overline{L} = \{Cy:y\in \mathbb{R}^m\}\cap \Z^n$. By turning $C^\top$ into Hermite normal form after applying elementary unimodular column operations, it can be shown that $[\overline{L}:L] = \gcd(C)$~(see \cite{Martinet03}, Proposition 1.1.5~(4)). Finally, observe that $L^\perp = \left\{x\in \Z^n:C^\top x=\0\right\}$. By combining the equalities provided thus far, we see that $\det(L^\perp) = \sqrt{\det(C^\top C)}/\gcd(C)$.

\paragraph{Lattice-free sets.} Observe that $P:=\{C\lambda:-\1<\lambda<\1\}$ is \emph{lattice-free}, that is, $P$ contains no nonzero vector from the lattice $L$. 
The set $P$ is the union of $2^m$ reflections of the fundamental parallelepiped of $L$, whose relative interiors are pairwise disjoint, so $\vol_m(P) = 2^m \det(L)$. Observe further that $P$ is a convex set that is symmetric about the origin. An important result of Minkowski from the Geometry of Numbers states that among all such sets which are lattice-free, $P$ reaches the maximum $m$-dimensional volume. 

\begin{MFT}[see \cite{Martinet03}, Theorem 2.7.1]
Let $L\subseteq \Z^n$ be an $m$-dimensional lattice. Let $C$ be a convex set that is symmetric about the origin. If $C\cap L=\{\0\}$, then $\vol_m(C)\leq 2^m\det(L)$.
\end{MFT}

\paragraph{Hadamard's inequality.} 
Let $B\in \mathbb{Z}^{m\times m}$ be a matrix with columns $b^1,\ldots,b^m$. Observe that $|\det(B)|$ is the volume of the parallelepiped with sides $b^1,\ldots,b^m$, implying in turn the well-known \emph{Hadamard inequality}, stating that $|\det(B)|\leq \prod_{i=1}^m \|b^i\|_2$ (\cite{Hadamard1893}, see~\cite{Beckenbach61}, Chapter 2, \S11). Subsequently, since $\|b^i\|_2\leq \sqrt{m}\|b^i\|_\infty$ for each $i\in [m]$, it follows that $
|\det(B)|^{1/m}\leq \sqrt{m}\|B\|_\infty.$

\paragraph{Minors of the adjugate matrix.} Let $U$ be an $n$-by-$n$ matrix. The \emph{adjugate of $U$}, denoted $\mathrm{adj}(U)$, is the $n$-by-$n$ matrix such that $U \mathrm{adj}(U) = \det(U) I$. It is well-known that for all $i,j\in [n]$, $\mathrm{adj}(U)_{ji}$ is a cofactor of $U$, namely, it is $(-1)^{i+j}$ times the determinant of the submatrix of $U$ obtained after removing row $i$ and column $j$. What is less known is an extension of this known attributed to Jacobi, which expresses every minor of $\mathrm{adj}(U)$ as a cofactor of the original matrix $U$.

\begin{JT}[see \cite{Prasolov94}, Theorem 2.5.2]
Let $U$ be an $n$-by-$n$ matrix, and $m\in [n-1]$. For row labels $I\subseteq [n]$ and column labels $J\subseteq [n]$ of $U$ such that $|I|=|J|=m$, we have 
$$
\det(\mathrm{adj}(U)_{J\times I}) = 
\mathrm{sgn}(\sigma)
\det(U_{\overline{I}\times \overline{J}}) \det(U)^{m-1}.
$$ where $\sigma:[n]\to [n]$ is the permutation defined as $\sigma(i_k) = j_k$ for all $k\in [n]$, for the following labelling of the rows and columns of $U$: the rows$/$columns of $U_{I\times J}$ are labelled $\{i_1,\ldots,i_m\}/\{j_1,\ldots,j_m\}$ from top to bottom$/$left to right, and the rows$/$columns of $U_{\overline{I}\times \overline{J}}$ are labelled $\{i_{m+1},\ldots,i_n\}/\{j_{m+1},\ldots,j_n\}$ from top to bottom$/$left to right.
\end{JT}

\end{document}